\documentclass[12pt]{journal}
\usepackage{amsmath,amsthm,amsfonts,amssymb}
\usepackage{array,fullpage, epic, eepic}
\usepackage{graphicx}
\usepackage{color}
\usepackage{float}
\usepackage{algorithm}
\usepackage{algorithmic}
\usepackage{subcaption}

\theoremstyle{plain}

\newtheorem{lem}{Lemma}[section]
\newtheorem{dfn}[lem]{Definition}
\newtheorem{theo}[lem]{Theorem}
\newtheorem{pro}[lem]{Proposition}
\newtheorem{cnj}[lem]{Conjecture}
\newtheorem{cor}[lem]{Corollary}
\newtheorem{exes}[lem]{Examples}

\theoremstyle{definition}

\newtheorem{ex}{Example}

\newcommand{\vp}{\varphi}
\newcommand{\lf}{\lfloor}
\newcommand{\rf}{\rfloor}

\newcommand{\N}{\mathbb{N}}

\def\DJ{{\hbox{D\kern-.8em\raise.15ex\hbox{--}\kern.35em}}}

\newcommand{\ps}{0.25}

\title{Fibonacci-like partitions and their associated piecewise-defined permutations}
\author{Jon Asier B\'arcena-Petisco, Luis Mart\'\i nez, Mar{\'\i}a Merino,\\ Juan Manuel Montoya, Antonio Vera-L\'opez}

\begin{document}

\maketitle

\begin{abstract}
In this paper we introduce a family of partitions of the set of natural numbers, Fibonacci-like partitions. In particular, we introduce a Fibonacci-like partition in a number of parts corresponding to the Fibonacci numbers, the standard Fibonacci-like partitions of the first kind. That partition refines the well-known partition in two parts associated to the Wythoff sequences. In a similar way, we introduce another family of Fibonacci-like partitions of the natural numbers, with an initial segment removed, in a number of parts which also follows the Fibonacci sequence, the standard Fibonacci-like partitions of the second kind.

We study some piecewise-defined permutations over Fibonacci-like partitions, which give both torsion-free and torsion elements in the symmetric group over the set of natural numbers.

We give an application of Fibonacci-like partitions to the study of a sequence analyzed by A. Shapovalov and by B.J. Venkatachala.

We give also a generalization of the mentioned sequence, involving techniques that are not related to Fibbonaci numbers nor Beatty sequences, obtaining an infinite family of sequences that we conjecture that produce an infinite orthogonal array in which the number of symbols, trials and factors is countable infinite, and having an automorphism group isomorphic to $\mathbb Z$ acting regularly on the symbol set.

\end{abstract}

\section{Introduction}\label{intr}

Beatty sequences are important in number theory and in many other areas of mathematics. Given a positive irrational number $\alpha$, the Beatty sequence $\mathfrak B_{\alpha}$ is defined as  $$\mathfrak B_{\alpha}=(\lfloor n\alpha\rfloor)_{n\in\mathbb N},$$ where $\lfloor x\rfloor$ denotes the floor of the number $x$ (be careful that some authors denote by $\mathfrak B_{1/\alpha}$ what we are denoting by $\mathfrak B_{\alpha}$). We will call range of a sequence to the set of its images, and if a sequence is given by a mapping $f$, we will denote its range by $\mathfrak R(f)$. Two well-known Beatty sequences are those corresponding to $\alpha=\varphi$ and $\alpha=\varphi^2$, where $\varphi=\frac{1+\sqrt{5}}{2}$ denotes the golden ratio, which will be denoted by this same symbol in the rest of the paper. They are called the lower and upper Wythoff sequences, respectively. We will do an extensive use of these two sequences in the sequel, and we will denote by $a(n)$ and $b(n)$, respectively, their general terms, that is, $a(n)=\lfloor n\varphi\rfloor$ and $b(n)=\lfloor n\varphi^2\rfloor$. Accordingly, $A$ and $B$ will denote the ranges $A=\mathfrak R(\mathfrak B_{\alpha})=\{1,3,4,6,8,9,\cdots\}=\{\lfloor \varphi n\rfloor\mid n\in\mathbb N\}$ and $B=\mathfrak R(\mathfrak B_{\alpha^2})=\{2,5,7,10,13,15,\cdots\}=\{\lfloor \varphi^2 n\rfloor\mid n\in\mathbb N\}$ of the lower and upper Wythoff sequences, respectively.

A well-known result, Beatty's theorem (\cite{beatty}), states that the complement of the range of a Beatty sequence is again the range of a Beatty sequence. More specifically, if $\alpha$ is an irrational number greater than $1$, then $\mathbb N-\mathfrak R(\mathfrak B_{\alpha})=\mathfrak R(\mathfrak B_{\alpha/(\alpha-1)})$. In other words, the ranges of the Beatty sequences corresponding to $\alpha$ and $\alpha/(\alpha-1)$ partition the set of natural numbers as a disjoint union of two subsets.

Uspensky's theorem (\cite{usp}) proves that this is the most we can get with Beatty sequences, that is, we cannot partition the set of natural numbers as a union of more than two pairwise disjoint ranges of Beatty sequences. Hildebrand and Li managed in (\cite{hildli}) to partition the set $\mathbb N$ into more than two parts by using a generalization of Beatty sequences that they called almost Beatty sequences. Moreover, Allouche and Dekking presented in (\cite{ad}) generalized Beatty sequences that gave partitions of the set of natural numbers in three parts.

Let us now explain in detail the contribution in our paper.

In section \ref{fibpart} we will present Fibonacci-like partitions, which are a very general family of partitions of the set of natural numbers.

More specifically, in Section \ref{sfibpart} we will introduce a family $\{\mathfrak P_n\}_{n\in\mathbb N}$ of partitions, the standard Fibonacci-like partitions of the first kind. In the $n$-th partition $\mathfrak P_n$ the set of natural numbers is partitioned in $F(n)$ parts, where $F(n)$ is the Fibonacci sequence defined by $F(1)=1,F(2)=1$ and $F(n)=F(n-1)+F(n-2)$ when $n\geq 3$. Although Fibonacci sequence can be extended in the obvious way to arbitrary integer values of $n$, we will need only in the sequel the value $F(0)=0$. We will introduce also in the same section another family $\{\mathfrak P^{\prime}_n\}_{n\in\mathbb N}$ of partitions of $\{m\in\mathbb N\mid m\geq F(n)\}$, the standard Fibonacci-like partitions of the second kind. In this case the partition $\mathfrak P^{\prime}_n$ also has $F(n)$ parts.

In Section \ref{permpart} we will study some permutations of the set of natural numbers that can be defined piecewise on blocks of Fibonacci-like partitions, showing that we can get both elements of finite and infinite order in the symmetric group.

In Section \ref{m2} we will give an application of standard Fibonacci-like partitions to study, by using an algorithm that we introduce in that section, a permutation of the set of nonnegative integers analyzed by A. Shapovalov (\cite{shap}) and by B.J. Venkatachala (\cite{venk}).

In Section \ref{ioas} we will extend the concept of orthogonal array to the case in which the alphabet $S$ and the set indexing the factors is infinite, generalizing thus the concept of orthogonal array to the one of infinite orthogonal array.

We will analyze in more detail the algorithm presented in the previous section, which we conjecture that conforms an infinite difference matrix. This difference matrix originates an infinite orthogonal array. These infinite OAs have an automorphism group acting regularly on the set of symbols. The conjectured difference matrix is non-trivial, in the sense that it cannot be originated from an infinite countable field.

Finally, in Section \ref{oq} we will comment some open questions with respect to these infinite orthogonal arrays and present some future lines of research.

\section{Fibonacci-like partitions of the natural numbers}\label{fibpart}

In this paper $\mathbb Z^{\geq 0}$ will denote the set of non-negative numbers, that is, $\mathbb Z^{\geq 0}=\{0,1,2,\ldots\}$ and $\mathbb N$ will denote the set of natural numbers, that is, $\mathbb N=\{1,2,\ldots\}$.

If $i\in\mathbb Z^{\geq 0}$ and $j\in\mathbb Z$ we will define the sequence $f_{i,j}$ by

$$f_{i,j}(n)=F(i+1)a(n)+F(i)n-j\text{ for }n\in\mathbb N.$$

Given $i$ and $j$ we will put $R_{i,j}=\mathfrak R(f_{i,j})$, that is, $R_{i,j}$ will denote the range of the sequence $f_{i,j}$.

These sequences $f_{i,j}$ are special cases of the generalized Beatty sequences presented by Allouche and Dekking in \cite{ad}.

The family of sequences $\{f_{i,j}\}$ contains as special cases the compositions of lower and upper Wythoff sequences studied by Carlitz, Scoville and Hoggat in (\cite{csh1}) and (\cite{csh2}) and rediscovered latter by Kimberling in (\cite{kim}).

Let us give some important examples:

\begin{exes}\label{exesfij}\hfill
\begin{enumerate}
\item $f_{0,0}=1, 3, 4, 6, 8, 9, 11, 12, 14, 16,\dots$ and $f_{1,0}=2, 5, 7, 10, 13, 15, 18, 20, 23, 26,\dots$ are the lower and upper Wythoff sequences, respectively (sequences A000201 and A001950 in Sloane's on-line encyclopedia of integer sequences \cite{sol}), whose ranges $R_{0,0}$ and $R_{1,0}$ form a partition of the set $\mathbb N$ of natural numbers.
\item $f_{0,1}=0, 2, 3, 5, 7, 8, 10, 11, 13, 15,\dots$ is the sequence A022342 in \cite{sol} formed by the integers with \lq even\rq\ Zeckendorf expansions.
\item $f_{0,-1}=2, 4, 5, 7, 9, 10, 12, 13, 15, 17,\dots$, when an item $1$ is added at the beginning, is the sequence with reference A026351 in \cite{sol}, and when an item $0$ is added at the beginning, is the sequence A004956. Those sequences are defined in terms of the floor and ceiling functions, respectively.
\item $f_{0,2}=-1, 1, 2, 4, 6, 7, 9, 10, 12, 14,\dots$, when the first item, $-1$, is removed, is the sequence A026273 in \cite{sol} whose $n$-th term is the least $k$ such that $s(k)=n$, where $s=$A026272.
\item $f_{0,-2}=3, 5, 6, 8, 10, 11, 13, 14, 16, 18,\dots$, when an item $2$ is added at the beginning, is the sequence A099267 in \cite{sol} formed by the numbers generated by the golden sieve.
\item $f_{0,3}=-2, 0, 1, 3, 5, 6, 8, 9, 11, 13,\dots$, when the first item, $-2$, is removed, is the sequence A058065 in \cite{sol}, which is referenced as the complement of A057843.
\item $f_{0,-3}=4, 6, 7, 9, 11, 12, 14, 15, 17, 19,\dots$, when the items $1$ and $2$ are added at the beginning, is the sequence A184732 in \cite{sol}, which is referenced as the complement of A184733.
\item $f_{1,1}=1, 4, 6, 9, 12, 14, 17, 19, 22, 25,\dots$ is the sequence A003622 in \cite{sol} formed by the Wythoff compound sequence $AA$.
\item $f_{1,-3}=5, 8, 10, 13, 16, 18, 21, 23, 26, 29,\dots$, when an item $1$ is added at the beginning, is the sequence A022413 in \cite{sol} formed by the Kimberling sums $K_n+K_2$.
\item $f_{2,0}=3, 8, 11, 16, 21, 24, 29, 32, 37, 42,\dots$ is the sequence A003623 in \cite{sol} formed by the Wythoff $AB$-numbers.
\item $f_{2,1}=2, 7, 10, 15, 20, 23, 28, 31, 36, 41,\dots$ is the sequence A035336 in \cite{sol} formed by the second column of Wythoff array.
\item $f_{2,-1}=4, 9, 12, 17, 22, 25, 30, 33, 38, 43,\dots$, is the sequence A089910 in \cite{sol} formed by the indices at which blocks $(1;1)$ occur in infinite Fibonacci word.
\item $f_{2,2}=1, 6, 9, 14, 19, 22, 27, 30, 35, 40,\dots$ is the sequence A134859 in \cite{sol} formed by the Wythoff $AAA$ numbers.
\item $f_{2,3}=0, 5, 8, 13, 18, 21, 26, 29, 34, 39,\dots$, when the first item, $0$, is removed, is the sequence A101345 in \cite{sol} whose $n$-th term is formed by Knuth's Fibonacci (or circle) product \lq $2\circ n$\rq.
\item $f_{2,-3}=6, 11, 14, 19, 24, 27, 32, 35, 40, 45,\dots$, when the item $3$ is added at the beginning, is the sequence A047924 in \cite{sol} whose $n$-th term is $b(a(n)+1)+1$, where $a$ and $b$ are the corresponding lower and upper Wythoff sequences.
\item $f_{2,4}=-1, 4, 7, 12, 17, 20, 25, 28, 33, 38,\dots$, when the first item, $-1$, is removed, is the sequence A190460 in \cite{sol} formed by the positions of $2$ in A190457.
\item $f_{3,0}=5, 13, 18, 26, 34, 39, 47, 52, 60, 68,\dots$ is the sequence A101864 in \cite{sol} formed by the Wythoff $BB$ numbers.
\item $f_{3,1}=4, 12, 17, 25, 33, 38, 46, 51, 59, 67,\dots$ is the sequence A134860 in \cite{sol} formed by the Wythoff $AAB$ numbers.
\item $f_{3,2}=3, 11, 16, 24, 32, 37, 45, 50, 58, 66,\dots$ is the sequence A035337 in \cite{sol} formed by the third column of Wythoff array.
\item $f_{3,3}=2, 10, 15, 23, 31, 36, 44, 49, 57, 65,\dots$ is the sequence A134861 in \cite{sol} formed by the Wythoff $BAA$ numbers.
\item $f_{3,-3}=8, 16, 21, 29, 37, 42, 50, 55, 63, 71,\dots$, when the item $3$ is added at the beginning, is the sequence A188012 in \cite{sol}, formed by the positions of $0$ in A188011.
\item $f_{3,5}=0, 8, 13, 21, 29, 34, 42, 47, 55, 63,\dots$, when the first item, $0$, is removed, is the sequence A101642 in \cite{sol} whose $n$-th term is formed by Knuth's Fibonacci (or circle) product \lq $3\circ n$\rq.
\item $f_{4,0}=8, 21, 29, 42, 55, 63, 76, 84, 97, 110,\dots$ is the sequence A134862 in \cite{sol} formed by the Wythoff $ABB$ numbers.
\item $f_{4,1}=7, 20, 28, 41, 54, 62, 75, 83, 96, 109,\dots$ is the sequence A134863 in \cite{sol} formed by the Wythoff $BAB$ numbers.
\item $f_{4,3}=5, 18, 26, 39, 52, 60, 73, 81, 94, 107,\dots$ is the sequence A035338 in \cite{sol} formed by the fourth column of Wythoff array.
\end{enumerate}
\end{exes}

We will state the theorem of bifurcation of the $R_{i,j}$ that will be the basis for Fibonacci-like partitions. We will give first a lemma:

\begin{lem}\label{basicrelwyth} If $n\in\mathbb N$ then
\begin{enumerate}
\item $a(a(n))=a(n)+n-1$.
\item $a(b(n))=2a(n)+n$.
\item $b(n)=a(n)+n$.
\end{enumerate}
\end{lem}

\begin{proof}
Parts (i) and (iii) of the lemma are (1.11) in (\cite{csv}). Part (ii) of the lemma is a consequence of (1.11) and (1.12) in (\cite{csv}).
\end{proof}

\begin{theo}\label{bifparts} If $i\in\mathbb Z^{\geq 0}$ and $j\in\mathbb Z$, then $$R_{i,j}=R_{i+1,F(i+1)+j}\cup R_{i+2,j},$$ where the union is a disjoint union.
\end{theo}

\begin{proof} It is a consequence of the two following identities: \begin{equation}\label{theoid1} f_{i+1,F(i+1)+j}(n)=f_{i,j}(a(n))\ \forall i\in \mathbb Z^{\geq 0},\forall j\in\mathbb Z,\forall n\in\mathbb N\end{equation} and \begin{equation}\label{theoid2} f_{i+2,j}(n)=f_{i,j}(b(n))\ \forall i\in \mathbb Z^{\geq 0},\forall j\in\mathbb Z,\forall n\in\mathbb N.\end{equation} Identity (\ref{theoid1}) is equivalent to: $$F(i+2)a(n)+F(i+1)n-F(i+1)-j=F(i+1)a(a(n))+F(i)a(n)-j,$$ that is, to $$F(i+1)a(n)+F(i+1)n-F(i+1)=F(i+1)a(a(n))$$ and, factoring out $F(i+1)$, the result follows from part (i) of the previous lemma.

With respect to Identity (\ref{theoid2}), it is equivalent to: \begin{equation}\label{fipt}F(i+3)a(n)+F(i+2)n-j=F(i+1)a(b(n))+F(i)b(n)-j.\end{equation} By part (ii) of the previous lemma, $a(b(n))=2a(n)+n$, and therefore (\ref{fipt}) is equivalent to: $$F(i+3)a(n)+F(i+2)n=2F(i+1)a(n)+F(i+1)n+F(i)b(n),$$ and hence to: $$F(i+2)a(n)+F(i)n=F(i+1)a(n)+F(i)b(n).$$ This is equivalent to: $$F(i) a(n)+F(i)n=F(i) b(n)$$ and, after factoring out $F(i)$, to $a(n)+n=b(n)$, which is true by part (iii) of the previous lemma.
\end{proof}

If we have a partition of $\mathbb N$ as a disjoint union of some $R_{i,j}$, which we will call $R_{i,j}$-partitions, we can obtain new partitions by applying Theorem \ref{bifparts} to some of the parts. The partition of the following example will be used later in Section \ref{permpart}:

\begin{ex} $R_{1,0},R_{2,0},R_{2,2},R_{3,1}$ is a partition of $\mathbb N$. To prove it, we consider first the classical partition $R_{0,0},R_{1,0}$ indicated in the first example of Examples \ref{exesfij}, formed by the ranges of the lower and upper Wythoff sequences. If we apply Theorem \ref{bifparts} with $i=0,j=0$ we obtain the partition $R_{1,0},R_{1,1},R_{2,0}$ and, if we apply again the theorem with $i=1,j=1$ we obtain the partition
$$R_{1,0}=\{2, 5, 7, 10, 13, 15, 18, 20, 23, 26,\dots\},$$
$$R_{2,0}=\{3, 8, 11, 16, 21, 24, 29, 32, 37, 42,\dots\},$$
$$R_{2,2}=\{1, 6, 9, 14, 19, 22, 27, 30, 35, 40,\dots\},$$
$$R_{3,1}=\{4, 12, 17, 25, 33, 38, 46, 51, 59, 67,\dots\}.$$ 
\end{ex}

\section{Standard Fibonacci-like partitions of the natural numbers}\label{sfibpart}

\begin{dfn}\label{dksflpfk} If $k\geq3$, we will call $k$-th standard Fibonacci-like partition of the first kind $\mathfrak P_k$ to 
$$\big\{R_{i,0},R_{i,1},\ldots,R_{i,F(i+2)-1},R_{i+1,0},R_{i+1,1},\ldots, 
R_{i+1,F(i+1)-1}
\big\},$$
for $i=k-3$. Also, we define $\mathfrak P_1=\mathfrak P_2=\N$.
\end{dfn}

\begin{theo}\label{fibonaccipart} If $k\in\mathbb N$, then the $k$-th standard partition $\mathfrak P_k$ is a partition in $F(k)$ parts of the set of natural numbers.
\end{theo}

Before proving the theorem we will note that for $k=1$ and $k=2$ the partitions $\mathfrak P_1$ and $\mathfrak P_2$ have only one part, which is the set $\mathbb N$, and this agrees with the initial conditions in the Fibonacci sequence, that is, $F(1)=F(2)=1$. For $k=3$, obviously, we obtain for $\mathfrak P_3$ a partition of $\mathbb N$ in two parts, which are the ranges of the lower and upper Wythoff sequences, respectively. More specifically, as we already mentioned, $R_{0,0}=A$ and $R_{1,0}=B$.

Now Theorem \ref{fibonaccipart} is a consequence of Theorem \ref{bifparts}. Just observe that each part in the partition $\mathfrak P_{k-1}$ bifurcates as a disjoint union of two parts in $\mathfrak P_k$.

The case when $k=4$ is specially interesting and we will use it later in connection with infinite OAs:

\begin{cor}\label{partt1t2t3} The sets $$R_{1,0}=\{2,5,7,10,13,15,18,20,23,26,\dots\},$$ $$R_{1,1}=\{1,4,6,9,12,14,17,19,22,25,\dots\}$$ and $$R_{2,0}=\{3,8,11,16,21,24,29,32,37,42,\dots\}$$ form a partition of $\mathbb N$.
\end{cor}

Now we will give an alternate inductive proof of Theorem \ref{fibonaccipart} different from the bifurcational argumentation that we have just given. We will use in the sequel the inductive relations when we introduce infinite orthogonal arrays in the next section.

To continue with, we prove a technical result that will be needed later on: 
\begin{lem}\label{newvarphim1}
We have the following equations:

\begin{equation}\label{eq:varphimo2}
\lfloor \varphi^{-1}\lfloor \varphi^2 n\rfloor\rfloor=\lfloor\varphi n\rfloor
\ \  \forall n\in\mathbb N.
\end{equation}
\begin{equation}\label{eq:varphimo3}
\lfloor \varphi^{-1}(\lfloor \varphi^2 n\rfloor+1)\rfloor=\lfloor\varphi n\rfloor
\ \  \forall n\in\mathbb N.
\end{equation}

\end{lem}

\begin{proof}
\begin{itemize}
\item Since $\vp^2=\vp+1$ and $\vp^{-1}=\vp-1$, equation \eqref{eq:varphimo2} is equivalent to 
$$a(a(n)+n)-a(n)-n=a(n),
$$
and this is an obvious consequence of parts (ii) and (iii) of Lemma \ref{basicrelwyth}. 
\item As for equation \eqref{eq:varphimo3}, we may prove it similarly to \eqref{eq:varphimo2}:
\begin{equation*}
\begin{split}
\lf\vp^{-1} (\lf\vp^2 n\rf+1)\rf-\lf\vp n\rf 
&=\lf \vp^{-1}(\lf\vp^2n\rf+1-\vp\lf\vp n\rf)\rf \\
&= \lf \vp^{-1}(n+1-\vp^{-1}\lf\vp n\rf)\rf\\
&=\lf \vp^{-1}(1+\vp^{-1}(\vp n-\lf\vp n\rf))\rf\\
&=0.
\end{split}
\end{equation*}
The second equality follows from part (iii) of Lemma \ref{basicrelwyth}, and the last one follows from $\vp n-\lf\vp n\rf\in (0,1)$ (since 
$\vp$ is irrational), and from $1+\vp^{-1}=\vp$.
\end{itemize}
\end{proof}

Before giving the alternate proof of Theorem \ref{fibonaccipart}, let us prove a technical result:

\begin{pro}\label{indpartfkind} If $i,j\in\mathbb Z^{\geq 0}$ then
\begin{enumerate}
\item $f_{i,j}(n)+1=f_{i,j-1}(n)\ \forall n\in\mathbb N$.
\item $f_{i+1,j}(n)+1=f_{i+1,j-1}(n)\ \forall n\in\mathbb N$.
\item $f_{i,0}(a(n))+1=f_{i+1,F(i+1)-1}(n)\ \forall n\in\mathbb N$.
\item $f_{i,0}(b(n))+1=f_{i,F(i+2)-1}(b(n)+1)\ \forall n\in\mathbb N$.
\item $f_{i+1,0}(n)+1=f_{i,F(i+2)-1}(a(n)+1)\ \forall n\in\mathbb N$.
\end{enumerate}
\end{pro}

\begin{proof}
The proof of (i) follows directly from the definitions, and (ii) is a special case of (i) when the first index is at least $1$.
 
In order to establish (iii) we have to prove:
\[F(i+1)a(a(n))+F(i)a(n)=F(i+2)a(n)+F(i+1)n-F(i+1),
\]
which is equivalent to 
\[F(i+1)(a(a(n))-a(n))=F(i+1)(n-1),
\]
which is true by part (i) of Lemma \ref{basicrelwyth}.

In order to establish (iv) we have to prove:
\[F(i+1)a(b(n))+F(i)b(n)=F(i+1)a(b(n)+1)+F(i)(b(n)+1)-F(i+2),
\]
which is equivalent to proving that
\[\label{cevp}\lf\vp\lf\vp^2n\rf-\lf\vp(\lf\vp^2n\rf+1)\rf+1\rf=0.
\]
If we put $\vp=1+\vp^{-1}$, this 
is equivalent to

\[\lf\vp^{-1}(\lf\vp^2n\rf+1)\rf=\lf\vp^{-1}\lf\vp^2n\rf\rf,\]

which follows from Lemma \ref{newvarphim1}.

Demonstrating (v) is equivalent to proving:
\[F(i+2)a(n)+F(i+1)n=F(i+1)a(a(n)+1)+F(i)(a(n)+1)-F(i+2).\]
This follows from:
\[a(a(n)+1)-n=a(n)+1,
\]
which follows from:
\begin{equation*}
\begin{split}
a(a(n)+1)-n&=\lf\vp(\lf\vp n\rf+1)\rf-n\\&=
\lf\vp n\rf+1+ \lf\vp^{-1}(\lf\vp n\rf+1)\rf-n\\&=
\lf\vp n\rf+1=a(n)+1,
\end{split}
\end{equation*}
where the penultimate inequality is an easy consequence of the fact that, if $u_n=\varphi n-\lf\varphi n\rf$, then $0<\varphi^{-1}(1-u_n)<1$.
\end{proof}

The following corollary of the previous proposition will be used later in another section when we study infinite orthogonal arrays:

\begin{cor}\label{cortripart}
The following identities hold:
$$f_{1,1}(n)+1=f_{1,0}(n)\ \forall n\in\N.$$
$$f_{1,0}(a(n))+1=f_{2,0}(n)\ \forall n\in\N.$$ $$f_{1,0}(b(n))+1=f_{1,1}(b(n)+1)\ \forall n\in\N.$$  $$f_{2,0}(n)+1=f_{1,1}(a(n)+1)\ \forall n\in\N.$$
\end{cor}

Now we will give the alternate proof of the theorem:

\begin{proof}[Proof of Theorem \ref{fibonaccipart}]
For $k=1,2$ the result is trivial. Let us now prove the result for $k\geq3$.
 We are going to prove by induction on $m$ that $m$ appears exactly once in the sets of $\mathfrak P_k$. 
\begin{itemize}
\item The base case $m=1$ is trivial. Indeed, $a(1)=1$, so we can prove easily that $1\in R_{i,F(i+2)-1}$, and that that is the only set in $\mathfrak P_k$ in which it appears. Indeed,
\[f_{i,F(i+2)-1}(1)=F(i+1)+F(i)-F(i+2)+1=1.
\]
To prove the uniqueness, it is worth remarking that:
\[f_{i+1,F(i+1)-1}(1)=F(i+2)+1>1,
\]
so $1$ is only in $R_{i,F(i+2)-1}$.
\item Let us suppose that $m-1$ appears in exactly one set and prove that $m$ appears in exactly one set:
\begin{itemize}
\item If $m-1$ belongs only to $R_{i,j}$ (for some $j\geq 1$), it suffices to use (i) in Proposition \ref{indpartfkind} to obtain that $m\in R_{i,j-1}$. Moreover, because Wythoff sequences are a partition and because of items (ii)-(v), we obtain that $R_{i,j-1}$ is the only set to which $m$ belongs. 
\item If $m-1$ belongs only to $R_{i+1,j}$ (for some $j\geq 1$), it suffices to use (ii) in Proposition \ref{indpartfkind}. to obtain that $m\in R_{i+1,j-1}$. Moreover, because Wythoff sequences are a partition and because of items (i) and (iii)-(v), we obtain that $R_{i+1,j-1}$ is the only set to which $m$ belongs. 
\item If $m-1$ belongs only to $R_{i,0}$, it suffices to use (iii) and (iv) in Proposition \ref{indpartfkind}, and the fact that Wythoff numbers are a partition to obtain that either $m\in R_{i+1,F(i+1)-1}$ or $m\in R_{i,F(i+2)-1}$. The uniqueness can be proved as in the previous two items.
\item If $m-1$ belongs only to $R_{i+1,0}$, it suffices to use (v) in Proposition \ref{indpartfkind} to obtain that $m\in R_{i,F(i+2)-1}$. The uniqueness can be proved as in the previous items. 
\end{itemize}
\end{itemize}
\end{proof}

Obviously, the bifurcations established in Theorem \ref{bifparts} can originate also other partitions of $\mathbb N$ different from the ones in Theorem \ref{fibonaccipart}. For instance, in $\mathfrak P_4$ we can bifurcate only $R_{1,0}=R_{2,1}\cup R_{3,0}$, so that we obtain the partition of $\mathbb N$ in the following four parts:

$$R_{2,1}=\{2,7,10,15,20,23,28,31,36,41,\dots\},$$

$$R_{3,0}=\{5,13,18,26,34,39,47,52,60,68,\dots\},$$

$$R_{1,1}=\{1,4,6,9,12,14,17,19,22,25,\dots\},$$

$$R_{2,0}=\{3,8,11,16,21,24,29,32,37,42,\dots\},$$

which is somewhere between $\mathfrak P_4$ and $\mathfrak P_5$.

This example, and the partitions in Theorem \ref{fibonaccipart}, among others, are partitions that originate from bifurcations of the sets $R_{i,j}$ and that, when viewed from the opposite point of view, can be contracted to the trivial partition of $\mathbb N$ with only one part, by doing fussions of the $R_{i,j}$. We will call fussionable to these partitions. In particular, the sequence of partitions in Theorem \ref{fibonaccipart} has no unbifurcated parts, and every part is bifurcated sooner than latter.

The building blocks $R_{i,j}$ can originate partitions in a more interesting and defiant way that through bifurcations, in the family of partitions that we will introduce here. In this case they will partition sets of the form $\{n\in\mathbb N\mid n\geq F(k)\}$ for a fixed natural number $k$, that is, the set of natural numbers greater than or equal to $F(k)$ will be the disjoint union of the parts.

\begin{dfn} If $k\in\mathbb N$, we will call $k$-th standard Fibonacci-like partition of the second kind $\mathfrak P^{\prime}_k$ to $\{R_{i,j}\}$ with $k-3\leq i\leq k-2$ and $-F(i+1)\geq j\geq 1-F(k)$
\end{dfn}

The proof of the following theorem and the subsequent corollary can be done in a similar way as for the standard Fibonacci-like partitions of the first kind, and will be omitted.

\begin{theo} Given $k\in\mathbb N$, if $k-3\leq i\leq k-2$ and $-F(i+1)\geq j\geq 1-F(k)$ then
\begin{enumerate}
\item $f_{i,j}(n)+1=f_{i,j-1}(n)$ for all $i,j$ with $k-3\leq i\leq k-2$ and $-F(i+1)\geq j\geq 2-F(k)$ and $\forall n\in\mathbb N$.
\item $f_{k-3,1-F(k)}(a(n))+1=f_{k-2,-F(k-1)}(n)\ \forall n\in\mathbb N$.
\item $f_{k-3,1-F(k)}(b(n))+1=f_{k-3,-F(k-2)}(b(n)+1)\ \forall n\in\mathbb N$.
\item $f_{k-2,1-F(k)}(n)+1=f_{k-3,-F(k-2)}(a(n)+1)\ \forall n\in\mathbb N$.
\end{enumerate}
\end{theo}

\begin{cor}\label{flip} A Fibonacci-like partition of the second kind gives a partition in $F(k)$ parts of the set $\{n\in\mathbb N\mid n\geq F(k)\}$.
\end{cor}

\begin{ex}\label{k5flpsk} For $k=5$, the Fibonacci-like partition of the second kind $\mathfrak P^{\prime}_5$ of $\{n\in\mathbb N\mid n\geq 5\}$ in $5$ parts is formed by the sets

$$R_{2,-2}=\{5, 10, 13, 18, 23, 26, 31, 34, 39, 44,\dots\},$$

$$R_{2,-3}=\{6, 11, 14, 19, 24, 27, 32, 35, 40, 45,\dots\},$$

$$R_{2,-4}=\{7, 12, 15, 20, 25, 28, 33, 36, 41, 46,\dots\},$$

$$R_{3,-3}=\{8, 16, 21, 29, 37, 42, 50, 55, 63, 71,\dots\},$$

$$R_{3,-4}=\{9, 17, 22, 30, 38, 43, 51, 56, 64, 72,\dots\},$$

and the following identity holds for the corresponding sequences:

$$f_{2,-2}(n)+1=f_{2,-3}(n),$$

$$f_{2,-3}(n)+1=f_{2,-4}(n),$$

$$f_{2,-4}(a(n))+1=f_{3,-3}(n),$$

$$f_{2,-4}(b(n))+1=f_{2,-2}(b(n)+1),$$

$$f_{3,-3}(n)+1=f_{3,-4}(n),$$

$$f_{3,-4}(n)+1=f_{2,-2}(a(n)+1).$$

\end{ex}

\section{Permutations based on $R_{i,j}$-partitions}\label{permpart}

We will consider piecewise permutations $\pi$ of the set of natural numbers based on $R_{i,j}$-partitions, which we will call $R_{i,j}$-permutations. The permutations will be defined piecewise on the $R_{i,j}$, and will be of the form $\lfloor (a\varphi +b)n+c\rfloor$ on that piece. We will codify the permutations as sets of $5$-tuples, where $(a,b,c,i,j)$ means $$\pi(n)=\lfloor (a\varphi +b)n+c\rfloor\text{ whenever }n\in R_{i,j}.$$ We can obtain many permutations in the symmetric group $S_{\mathbb N}$ in that way, both torsion elements (that is, of finite order) and torsion free elements (that is, of infinite order). We will show also, for convenience of the reader, the first $20$ terms of the $R_{i,j}$-permutations that we will consider.

In this section, juxtaposition of mappings will denote its composition.

Although the proof of the theorems in this section can be done using standard identities, they would be long and tedious for the reader, and therefore we have done them using the automatic prover software Walnut \cite{walnut}, made by Hamoon Mousavi, which is a free software for deciding first-order statements phrased in the Bucchi arithmetic, which is an extension of Presburger arithmetic. The reader can find more information about the use of Walnut to prove properties of sequences in \cite{shall1} and \cite{shall2}. The corresponding scripts can be found in the repository at \cite{ehubox}. The name of the script for each proof is indicated after each statement. When running the programs, the file \lq phin.txt\rq\ in the mentioned repository, which defines the automata to calculate the lower Wythoff sequence, must be in the \lq Automata Library\rq\ directory of the Walnut installation address.

In the evaluation of the $\pi(n)$ with Walnut when $a<0$ we have taken into account that, if $n\in\mathbb Z^{\geq 0}$, then $\varphi(-n)=-\varphi(n)+1$.

We will begin with some torsion elements of orders $2,4,6$.

\begin{theo}
The sequence $f=2,1,5,7,3,10,4,13,15,6,18,20,8,23,9,26,28,11,31,12,\dots$ defined by 

$$f(n)=\begin{cases} \lfloor \varphi n\rfloor +1, & \text{if }n\in R_{0,0}\\ \lfloor (\varphi -1)n\rfloor, & \text{ if }n\in R_{1,0}\end{cases}$$

is a $R_{i,j}$-permutation of order $2$.
\end{theo}

Proof: file walnut\textunderscore script\textunderscore theorem\textunderscore 4.1.txt\newline

With the notation of $5$-tuples $f$ is determined by $(1,0,1,0,0),(1,-1,0,1,0)$.\newline

This sequence $f$ will play an important role in the next sections.

\begin{theo}The sequence $g=2,1,4,3,6,5,9,12,7,14,17,8,19,10,22,25,11,27,13,30,\dots$ defined by the $5$-tuples

$(1,0,0,2,0),(1,0,-2,1,0),(1,-1,2,2,2),(1,-1,1,3,1)$

is a $R_{i,j}$-permutation of order $2$.
\end{theo}

Proof: file walnut\textunderscore script\textunderscore theorem\textunderscore 4.2.txt\newline

\begin{theo} The sequence $h=2,3,4,1,8,5,11,12,7,16,17,6,21,10,24,25,9,29,13,32,\dots$ defined by

$(1,0,0,2,0),(1,0,0,1,0),(1,-1,2,2,2),(1,-1,-1,3,1)$

is a $R_{i,j}$-permutation, and it is of order $4$.

Its powers are $R_{i,j}$-permutations defined by

$h^2=3,4,1,2,12,8,17,6,11,25,9,5,33,16,38,14,7,46,21,51,\dots$

$(1,1,-1,1,0),(0,1,2,2,2),(0,1,-2,2,0),(-1,2,1,3,1)$.

$h^3=h^{-1}=4,1,2,3,6,12,9,5,17,14,7,8,19,25,22,10,11,27,33,30,\dots$

$(1,0,3,2,2),(1,0,-2,1,0),(1,-1,1,3,1),(1,-1,1,2,0)$.

\end{theo}

Proof: file walnut\textunderscore script\textunderscore theorem\textunderscore 4.3.txt\newline

\begin{theo} The sequence $i=2,3,1,8,11,5,16,4,7,24,6,21,32,10,37,9,29,45,13,50,\dots$ defined by

$(1, 1, -2, 1, 0), (1, 0, 2, 3, 1), (1, -1, 2, 2, 2), (1, -1, 0, 2, 0)$

is a $R_{i,j}$-permutation, and it is of order $6$.

Its powers are $R_{i,j}$-permutations defined by

$i^2=3,1,2,4,6,11,9,8,16,14,5,12,19,24,22,7,17,27,32,30,\dots$

$(1,0,2,2,2),(1,0,-2,1,0),(0,1,0,4,0),(0,1,0,3,1),(-1,2,1,3,2)$.

$i^3=1,2,3,8,5,6,7,4,9,10,11,21,13,14,15,16,29,18,19,20,\dots$

$(1,0,2,3,1),(1,-1,0,4,0),(0,1,0,3,2),(0,1,0,2,2),(0,1,0,1,0)$.

$i^4=2,3,1,4,11,5,16,8,7,24,6,12,32,10,37,9,17,45,13,50,\dots$

$(1,1,-2,1,0),(1,-1,2,2,2),(1,-1,0,3,2),(0,1,0,4,0),(0,1,0,3,1)$.

$i^5=i^{-1}=3,1,2,8,6,11,9,4,16,14,5,21,19,24,22,7,29,27,32,30,\dots$

$(1,0,2,1,1),(1,0,-2,1,0),(1,-1,0,4,0),(-1,2,1,3,2)$.
\end{theo}

Proof: file walnut\textunderscore script\textunderscore theorem\textunderscore 4.4.txt\newline

Now we will give an element of infinite order.

\begin{theo} The sequence $j=1,2,5,3,7,4,10,13,6,15,18,8,20,9,23,26,11,28,12,31,\dots$ defined by

$(1,0,1,2,0),(1,0,-1,1,0),(1,-1,1,1,1)$

is a $R_{i,j}$-permutation.

Its inverse $j^{-1}=1,2,4,6,3,9,5,12,14,7,17,19,8,22,10,25,27,11,30,13,\dots$ is a $R_{i,j}$-permutation defined by

$(1,0,0,2,0),(1,0,0,1,1),(1,-1,1,2,1),(1,-1,0,3,0)$.

\end{theo}

Proof: file walnut\textunderscore script\textunderscore theorem\textunderscore 4.5.txt\newline

\begin{lem} If $n\in R_{1,0}$, then $j(n)\in R_{1,0}$ and $j^3(n)=j^2(n)+j(n)-2$.
\end{lem}

Proof: file walnut\textunderscore script\textunderscore lemma\textunderscore 4.6.txt\newline

\begin{theo}\label{g3} If $n\geq 1$, then $j^n(3)=F(n+3)+2$.
\end{theo}

\begin{proof}It can be easily proved by induction on $n$ using the previous lemma.
\end{proof}

Therefore, the permutation $j$ is of order $\infty$ in the permutation group generated by $j$.

\begin{theo}\label{fj} $fj=j^{-1}f$
\end{theo}

Proof: file walnut\textunderscore script\textunderscore theorem\textunderscore 4.8.txt\newline

By Theorems \ref{g3} and \ref{fj}, the permutation group generated by $f$ and $j$ is an infinite dihedric group.
                                                                    
\section{An application of Fibonacci-like partitions}\label{m2}

If $A\subseteq\mathbb Z^{\geq 0}$ we will denote, as usual in the literature, by $\text{mex}(A)$ (minimum excluded value) the number $\min (\mathbb Z^{\geq 0}-A)$, that is, the minimum non-negative integer that is not in $A$.

In this section and in the following one, we will consider the following greedy algorithm:\newline\newline

\begin{algorithm}[H]
\caption{}\label{alg}
\begin{algorithmic}[1]
\REQUIRE $n,m\in\mathbb N$
\ENSURE $Q^{[m,n]}\in M_{(m+1)\times n}(\mathbb Z^{\geq 0})$
\FOR{$i\in\{0,\dots,m\}$}
\FOR{$j\in\{0,\dots,n-1\}$}
\STATE $S=\{q^{[m,n]}_{i,k}+q^{[m,n]}_{l,j}-q^{[m,n]}_{l,k}\mid 0\leq l<i,0\leq k<j\}$
\STATE $q^{[m,n]}_{i,j}=\text{mex}(S\cap\mathbb Z^{\geq 0})$
\ENDFOR
\ENDFOR
\RETURN $Q^{[m,n]}$
\end{algorithmic}
\end{algorithm}

Let us see an example for $m=3,n=4$. Obviously, $S=\emptyset$ when $i=0$, and also when $j=0$. We will show next $S$ in the other cases:

$$\begin{tabular}{c|ccc}
i/j & 1 & 2 & 3 \\
\hline
1 & \{0\} & \{0,1\} & \{0,1,2\} \\
2 & \{0,1\} & \{0,2,3\} & \{0,1,2,3,4\} \\
3 & \{0,1,2\} & \{0,1,2,3,4\} & \{0,3,5,6,9\}   
\end{tabular}
$$

Hence, the matrix is $$\left(\begin{array}{cccc}
0 & 0 & 0 & 0\\
0 & 1 & 2 & 3\\
0 & 2 & 1 & 5\\
0 & 3 & 5 & 1
\end{array}\right).$$

We observe that the matrices $Q^{[m,n]}$ obtained with Algorithm \ref{alg} are monotone with respect to the inclusion, in the sense that if $m\leq m^{\prime},n\leq n^{\prime}$, then $Q^{[m,n]}$ is an initial submatrix of $Q^{[m^{\prime},n^{\prime}]}$, that is, $Q^{[m,n]}_{i,j}=Q^{[m^{\prime},n^{\prime}]}_{i,j}$ $\forall i\in\{0,\dots,m\}\forall j\in\{0,\dots,n-1\}$. This let us to obtain an infinite matrix $Q$ with an infinite number of rows and columns and obtain, for each $m\in\mathbb Z^{\geq 0}$, an infinite matrix $Q^{[m]}$ with $m+1$ rows and an infinite number of colums.

We will study in this section the infinite matrix $Q^{[2]}$. Its first row is, obviously, the all-$0$ sequence, because the corresponding sets $S$ in the algorithm are $\emptyset$. Its second row has $j$ in its $j$-th position, because it is clear that $S=\{0,1,\dots,j-1\}$.

We will determine in what follows its third row. If $j\in\mathbb Z^{\geq 0}$, we will put $S_j$ to be the set of avoided values obtained with Algorithm 1 for $i=2$ and the corresponding value of $j$. We will put also $q_j=q^{[2]}_{2,j}$, $U_j=\{q_k\mid 0\leq k<j\}$,$W_j=\{q_k+j-k\mid 0\leq k<j\}$. If $j\in\mathbb N$, let $n(j)=\max\{k\mid \{0,1,\dots,k\}\subseteq U_j\}$. By the definition of $n(j)$, for every $j\in\mathbb N$, $$U_j=\{0,1,\dots,n(j),z_{j,1},\dots,z_{j,j-n(j)-1}\},$$ where $z_{j,j-n(j)-1}>\dots >z_{j,1}\geq n(j)+2$. If $j\in\mathbb N$, we will put 
$m(j)=\lfloor\varphi j\rfloor-j+1$.

We will define also, for $n\in\mathbb N$, $f(n)=b(n),g(n)=b(n)-1$ and $h(n)=2a(n)+n$, and we define the following subsets of $\mathbb N$:
$$T_1=\{f(n)\mid n\in\mathbb N\}=\{2,5,7,10,13,15,18,20,23,26,\dots\},$$

$$T_2=\{g(n)\mid n\in\mathbb N\}=T_1-1=\{1,4,6,9,12,14,17,19,22,25,\dots\},$$

$$T_3=\{h(n)\mid n\in\mathbb N\}=\{3,8,11,16,21,24,29,32,37,42,\dots\}.$$

Observe that the mappings $f,g,h$ are just $f_{1,0},f_{1,1},f_{2,0}$, and $T_1,T_2,T_3$ are the corresponding sets $R_{1,0},R_{1,1},R_{2,0}$ appearing in Corollary \ref{partt1t2t3}, which form a partition of the set of natural numbers.

Given a set $X=\{x_1,\dots,x_p\}$ of nonnegative integer numbers, and given $y\in\mathbb Z^{\geq 0}$, $X+y$ will denote the set $\{x_1+y,\dots,x_p+y\}$.

As usual, $\lfloor x\rfloor$ will denote the integer part of the number $x$.

\begin{theo}\label{tqj} The sequence $\{q_{j}\}_{j\in \mathbb Z^{\geq 0}}$ is a permutation of the set $\mathbb Z^{\geq 0}$ and every integer number appears exactly once in the sequence $\{q_{j}-j\}_{j\in\mathbb Z^{\geq 0}}$.
\end{theo}

With the terminology that we will give in Section \ref{ioas}, the previous theorem states that $Q^{[2]}$ is an infinite cyclic difference matrix.

\begin{lem}\label{varphim1}$$\lfloor (\varphi-1)\lfloor \varphi^2 k\rfloor\rfloor-\lfloor\varphi^2 k\rfloor=-k\ \forall k\in\mathbb N$$
\end{lem}

\begin{proof} It holds that $\lfloor (\varphi-1)\lfloor\varphi^2 k\rfloor\rfloor-\lfloor\varphi^2(k)\rfloor=\lfloor\varphi\lfloor\varphi^2 k\rfloor\rfloor-2\lfloor\varphi^2 k\rfloor=a(b(k))-2 b(k)$.

By part (ii) of Lemma \ref{basicrelwyth}, this is $2a(k)+k-2b(k)$.

By part (iii) of the same lemma, this is $-k$, thus concluding the proof.

\end{proof}

\begin{lem}\label{ajajmu}
\begin{enumerate}
\item If $j-1\in A$, then $m(j)=m(j-1)+1$ and $\lfloor\varphi j\rfloor=\lfloor\varphi (j-1)\rfloor+2$.
\item If $j-1\in B$, then $m(j)=m(j-1)=\lfloor (\varphi-1)(j-1)\rfloor+1$ and $\lfloor\varphi j\rfloor=\lfloor\varphi (j-1)\rfloor+1$.
\end{enumerate}
\end{lem}

\begin{proof} If $j-1\in A$, then $j-1=a(k)$ for some $k\in\mathbb N$, and $m(j)=a(a(k)+1)-a(k), m(j-1)=a(a(k))-a(k)+1$. We have to prove that $a(a(k)+1)=a(a(k))+2$ which, by part (i) of Lemma \ref{basicrelwyth}, is equivalent to $a(a(k)+1)=a(k)+k+1$. By part (v) of Proposition \ref{indpartfkind} with $i=0$, we have $b(k)+1=a(a(k)+1)$ and now, by part (iii) of Lemma \ref{basicrelwyth}, $b(k)+1=a(k)+k+1$, which finishes the proof that $m(j)=m(j-1)+1$. Now this implies that $\lfloor\varphi j\rfloor=\lfloor\varphi (j-1)\rfloor+2$.

Similarly, if $j-1\in B$, then $j-1=b(k)$ for some $k\in\mathbb N$, and we have to prove that $a(b(k)+1)=a(b(k))+1$ and this, by part (ii) of Lemma \ref{basicrelwyth}, is equivalent to $a(b(k)+1)=2a(k)+k+1$. By part (iv) of Proposition \ref{indpartfkind} with $i=0$, we get $a(b(k)+1)=a(b(k))+1$ and, by part (ii) of Lemma \ref{basicrelwyth}, $a(b(k))=2a(k)+k$, and this proves that $m(j)=m(j-1)$.

Now, obviously, $m(j-1)=\lfloor\varphi (j-1)\rfloor-j+1+1=\lfloor (\varphi-1)(j-1)\rfloor+1$, and $\lfloor\varphi j\rfloor=\lfloor\varphi (j-1)\rfloor+1$ is a direct consequence of $m(j)=m(j-1)$.

\end{proof}

\begin{lem}\label{ahb}\hfill
\begin{enumerate}
\item $m(f(a(k)))=b(k)\ \forall k\in\mathbb N$.
\item $m(f(b(k)))+1=b(a(k)+1)\ \forall k\in\mathbb N$.
\item $m(g(a(k)))+1=b(k)\ \forall k\in\mathbb N$.
\item $m(g(b(k)))+2=b(a(k)+1)\ \forall k\in\mathbb N$.
\item $m(h(a(k)))+3=b(a(k)+1)\ \forall k\in\mathbb N$.
\item $m(h(b(k)))+2=b(b(k)+1)\ \forall k\in\mathbb N$.
\end{enumerate}
\end{lem}

\begin{proof} With respect to part (i), $m(f(a(k)))=m(b(a(k)))=\lfloor (\varphi-1)b(a(k))\rfloor+1$. By Lemma \ref{varphim1}, this is $-a(k)+b(a(k))+1$. By part (iii) of Lemma \ref{basicrelwyth}, this is $a(a(k))+1$, and by part (i) of the same lemma, this is $a(k)+k$. Finally, by part (iii) again, this is $b(k)$.

The proof of parts (ii) to (vi) is similar to the one of part (i) and will be omitted.
\end{proof}

\begin{lem}\label{mainlemmaq2} The following identities hold for every $j\in\mathbb Z^{\geq 0}$:
\begin{enumerate}
\item If $j\geq 1$, then $$W_j=\{m(j),m(j)+1,\dots,\lfloor\varphi j\rfloor\}$$
\item If $j\geq 1$, then $z_{j,i}=b(k_j+i-1)\ \forall i\in\{1,\dots,j-n(j)-1\}$, where $k_j$ is the first natural number such that $b(k_j)>n(j)$.
\item If $j\geq 1$ and $j\in T_1$, then $n(j)=m(j)-2$.
\item If $j\geq 1$ and $j\in T_2$, then $n(j)=m(j)-1$.
\item If $j\geq 1$ and $j\in T_3$, then $n(j)=m(j)$.
\item $$S_j=\begin{cases} \emptyset, & \text{if }j=0\\
\{0,1,\cdots,\lfloor \varphi j\rfloor\}, & \text{if }j\in A\\ \{0,1,\cdots,\lfloor\varphi j\rfloor\}-\{\lfloor (\varphi -1)j\rfloor\}, & \text{ if }j\in B\end{cases}$$
\item $$q_j=\begin{cases} 0, & \text{if }j=0\\
\lfloor \varphi j\rfloor +1, & \text{if }j\in A\\ \lfloor (\varphi -1)j\rfloor, & \text{ if }j\in B\end{cases}$$
\end{enumerate}
\end{lem}

\begin{proof}We will prove it by induction on $j$. The results hold trivially for $j=0$ and $j=1$ (having into account that there is no $i$ in (ii) for $j=1$). Let us suppose that it is true for $j-1$, and let $j\geq 2$.

By Algorithm $1$, $$S_j=U_j\cup W_j.$$ 

We have that $$W_j=\{q_0+j,q_1+j-1,\dots,q_{j-1}+1\}=(W_{j-1}+1)\cup\{q_{j-1}+1\},$$ and therefore $$W_j=\begin{cases} \{m(j-1)+1,\dots,\lfloor \varphi (j-1)\rfloor+1\}\cup\{\lfloor\varphi (j-1)\rfloor+2\}, & \text{if }j-1\in A\\ \{m(j-1)+1,\dots,\lfloor \varphi (j-1)\rfloor+1\}\cup\{\lfloor(\varphi-1) (j-1)\rfloor+1\}, & \text{if }j-1\in B\end{cases}$$ Now, (i) follows from Lemma \ref{ajajmu}.

We have $$U_{j-1}=\{0,1,\dots,n(j-1),z_{j-1,1},\dots,z_{j-1,j-n(j-1)-2}\}$$
To prove parts from (ii) to (v) we will distinguish several cases:

Case 1: If $j-1\in f(A)$ then $j-1\in T_1$ and $n(j-1)=m(j-1)-2$. By the definition of $f$ we have that $j-1\in B$, and therefore $q_{j-1}=\lfloor (\varphi-1)(j-1)\rfloor$.

Now, $$U_j=\{0,1,\dots,m(j-1)-2,z_{j-1,1},z_{j-1,2},\dots\}\cup\{\lfloor (\varphi-1)(j-1)\rfloor\}$$ By part (i) of Lemma \ref{ahb}, using the induction hypothesis we obtain that $z_{j-1,1}=n(j-1)+2$. By Lemma \ref{ajajmu} $\lfloor (\varphi-1)(j-1)\rfloor=m(j-1)-1$, and therefore  $$U_j=\{0,1,\dots,m(j),z_{j-1,2},\dots\},$$ which proves (ii). Now, since by Corollary \ref{cortripart}, $j\in T_3$, and $n(j)=m(j)$, we conclude that (v) holds.

Since $j\in A$, parts (vi) and (vii) are a consequence of what we have already proved.

The following remaining $5$ cases can be proved in a similar way to case $1$ using parts (ii) to (vi), respectively, of Lemma \ref{ahb}, and we will omit the proof:

Case 2: $j-1\in f(B)$.

Case 3: $j-1\in g(A)$.

Case 4: $j-1\in g(B)$.

Case 5: $j-1\in h(A)$.

Case 6: $j-1\in h(B)$.
\end{proof}

Now we will prove the theorem:

\begin{proof} Algorithm \ref{alg} guarantees that all the $q_j$ are distinct, and parts (iii),(iv) and (v) of Lemma \ref{mainlemmaq2} prove that every non-negative integer is in some $U_j$ and therefore is in the range of the sequence.

Thus, the $\{q_{j}\}_{j\in \mathbb Z^{\geq 0}}$ form a partition of $\mathbb Z^{\geq 0}$.

The identity \begin{equation}\label{comp1}\lfloor\varphi\lfloor\varphi k\rfloor\rfloor+1-\lfloor\varphi k\rfloor=k\ \forall k\in\mathbb N\end{equation} can be deduced from the case $j=0$ in Corollary 2 in (\cite{griff}) (in fact, the mentioned corollary is stated for $j\in\mathbb N$, but the proof can be adapted to cover also the case $j=0$).

We have by Lemma \ref{varphim1} the identity: \begin{equation}\label{comp2}\lfloor (\varphi-1)\lfloor \varphi^2 k\rfloor\rfloor-\lfloor\varphi^2 k\rfloor=-k\ \forall k\in\mathbb N.\end{equation}

Now, since $q_0-0=0$, the last part of the theorem follows from (\ref{comp1}) and (\ref{comp2}) using part (vii) of Lemma \ref{mainlemmaq2}.
\end{proof}

An alternate way to prove that $q_n$ is a permutation of $\mathbb Z^{\geq 0}$ is to use the following proposition, which prove that $q_n$ is an involution, that is, that $q_{q_n}=n\ \forall n\in \mathbb Z^{\geq 0}$, whose proof is an immediate consequence of Lemma \ref{basicrelwyth} and will be omitted:

\begin{pro}\begin{enumerate}
\item If $j\in A$, then $\lfloor(\varphi-1)(\lfloor\varphi j\rfloor+1)\rfloor=j$.
\item If $j\in B$, then $\lfloor\varphi\lfloor (\varphi-1)j\rfloor\rfloor+1=j$.
\end{enumerate}
\end{pro}

The sequence $q_n$ corresponds with sequence A002251 in Sloane's on-line encyclopedia of integer sequences \cite{sol}, obtained by swapping $a(k)$ and $b(k)$ for all $k\geq 1$. This is evident from the expression for $q_n$ given in Lemma \ref{mainlemmaq2}, and this clarifies the fact, that we already proved, that it is an involution. When we take $q_n+1$ we obtain the sequence studied on page 76 in \cite{berconguy}, which is also studied by  A. Shapovalov (\cite{shap}) and by Venkatachala (\cite{venk}) (sequence A019444 in Sloane's database). Interestingly, the property of the sequence $\{q_j\}_{j\in\mathbb Z^{\geq 0}}$ with the non-negative numbers as domain established in Theorem \ref{tqj} is shared by the mentioned sequence A019444 with the natural numbers as domain, as noted by Ivan Neretin in the comments to the sequence in Sloane's database, where he says that it is a lexicographically minimal sequence of distinct positive integers such that all values of $a(n)-n$ are also distinct.

\section{Infinite orthogonal arrays}\label{ioas}

Recall (\cite[Definition 1.1]{heslst}) that, if $S$ is a finite alphabet with $s$ symbols, and if $N,k,\lambda\in\mathbb N$, and $t$ is a non-negative integer with $t\leq k$, then an orthogonal array (abbreviated as $OA$) $A$ with $N$ runs $s$ levels, $k$ factors, strength $t$ and index $\lambda$ is an $N\times k$ array with entries in $S$ satisfying that every $N\times t$ subarray of $A$ contains each $t$-tuple based on $S$ exactly $\lambda$ times as a row.

Obviously, in that case the identity $N=\lambda s^t$ holds.

Such an array is said to be an $OA(N,k,s,t)$. When $\lambda=1$ the orthogonal array is said to have index unity. If, besides, $t=2$, we will denote an $OA(n^2,m,n,2)$, for brevity, just as an $OA(m,n)$.

The following array is an example of an $OA(4,3)$.

$$\left(
\begin{array}{cccc}
 0 & 0 & 0 & 0 \\
 0 & 1 & 1 & 2 \\
 0 & 2 & 2 & 1 \\
 1 & 0 & 2 & 2 \\
 1 & 1 & 0 & 1 \\
 1 & 2 & 1 & 0 \\
 2 & 0 & 1 & 1 \\
 2 & 1 & 2 & 0 \\
 2 & 2 & 0 & 2 \\
\end{array}
\right)$$

Symmetries of orthogonal arrays are considered through the introduction of automorphisms. Symbol automorphisms are permutations of symbols that preserve the $OA$ when the rows are reordered, and columns automorphisms are permutations of columns that preserve the $OA$ when the rows are reordered. More generally compositions of symbol automorphisms and columns automorphisms can be considered. The automorphisms of an $OA$ form a group with respect to composition, which is called the whole automorphisms group. It is useful to take also subgroups of the whole automorphisms group, which are referred as automorphisms groups. We will consider only automorphisms groups formed by symbol automorphisms.

A useful tool to obtain $OAs$ is the utilization of difference matrices:

\begin{dfn}\cite[Chapter 17, VI]{cd2007} Given a group $(G,+)$ of order $n$ and two natural numbers $m,\lambda$, a $(n,m;\lambda)$-difference matrix is a $m\times \lambda n$ matrix $D=(d_{ij})$ with entries in $G$ satisfying that for each $1\leq r<s\leq m$, the multiset $$\{d_{rl}-d_{sl}\mid 1\leq l\leq \lambda n\}$$ contains each element of $G$ exactly $\lambda$ times.
\end{dfn}

Obviously, if $D$ is a $(n,m;\lambda)$-difference matrix, then the $\lambda n^2\times m$ array $A$ with its rows indexed by pairs $(j,g)$ with $1\leq j\leq \lambda n$ and $g\in G$ and columns indexed by $i$ with $1\leq i\leq m$ given by $A=(a_{(j,g)i})=d_{ij}+g$ is a $OA(\lambda n^2,m,n,2)$.

For each $g\in G$ the mapping $f_g:G\longrightarrow G$ defined by $f(g^{\prime})=g^{\prime}+g$ is clearly an automorphism of the $OA$, and the set $\{f_g\mid g\in G\}$ is an automorphisms group isomorphic to $G$, and in this way the group $G$ acts regularly on the set of symbols.

We are specially interested in the case when $\lambda=1$, which is the one that we will consider in this work. In this way, we can obtain from any $(n,m;1)$-difference matrix an $OA(m,n)$ with an automorphism group isomorphic to $G$ acting regularly on their symbols. 

The case when the group $G$ is cyclic is specially interesting, and in this case difference matrices are called cyclic difference matrices (\cite{afg},\cite{ag},\cite{bj},\cite{cly},\cite{ge1},\cite{ge2},\cite{gl},\cite{lo},\cite{mk1},\cite{mk2},\cite{wj}).

We can generalize the definition of orthogonal array letting the alphabet $S$ to be a countable infinite set, so that the array has also a countable infinite number of rows, and letting also a countable infinite number of factors, so that the array has a countable infinite number of columns.

We can use symmetry in the search of countably infinite orthogonal arrays, by generalizing in an obvious way the concept of difference matrix:

\begin{dfn} Given a countable infinite group $(G,+)$ and a natural number $\lambda$, a $(\infty,\infty;\lambda)$-difference matrix is a matrix $D=(d_{ij})$ with its rows and columns indexed by the set of natural numbers with entries in $G$ satisfying that for each $r,s\in\mathbb N$ with $r<s$, the multiset $$\{d_{rl}-d_{sl}\mid l\in\mathbb N\}$$ contains each element of $G$ exactly $\lambda$ times.
\end{dfn}

We can use countably infinite difference matrices, using the same construction that the one for finite difference matrices, to obtain countably infinite orthogonal arrays.

An obvious way to get infinite difference matrices is the following one: Take a countably infinite field $K=\{a_0,a_1,\dots\}$ (here, obviously, the ordering of its elements is arbitrary and it is given just to stress the countability of its elements). We can suppose without losing generality that $a_0=0$ and hence that the $a_i$ with $i\in\mathbb N$ are inversible. Clearly, the infinite matrix $D=(d_{ij})$ with $i,j\in\mathbb N$ with entries in $K$ with $d_{ij}=a_ia_j$ is a $(\infty,\infty;1)$-difference matrix, because $d_{rl}-d_{sl}=(a_r-a_s)a_l$ and, since $a_r-a_s\not=0$, every element in $K$ can be expressed in that form for a unique index $l$.

In the previous construction we obtain infinite orthogonal arrays with the additive group of $K$ as a regular automorphism group, but it is not a cyclic group. As we said before, difference matrices over cyclic groups are specially investigated, and it would be desirable to consider cyclic groups in this infinite setting. All infinite cyclic groups are isomorphic to the additive group of integer numbers, and therefore this will be our focus of interest in what follows.

We conjecture that the following generalization of Theorem \ref{tqj} holds:

\begin{cnj}\label{conj} If we consider the non-negative integers $q_{i,j}=q_{i,j}^{[m,n]}$ obtained in Algorithm \ref{alg} (note that they are independent of $m$ and $n$), then for any $i\geq 1$ the sequence $\{q_{i,j}\}_{j\in\mathbb Z^{\geq 0}}$ is a permutation of the set $\mathbb Z^{\geq 0}$ and, for every natural numbers $i_1,i_2$ with $i_1<i_2$ every integer number appears exactly once in the sequence $\{q_{i_1,j}-q_{i_2,j}\}_{j\in\mathbb Z^{\geq 0}}$.
\end{cnj}

If Conjecture \ref{conj} were true then the infinite array $(q_{i,j})_{i\in\mathbb N,j\in\mathbb Z^{\geq 0}}$ would be a $(\infty,\infty;1)$-difference matrix, and we would obtain from it, in a similar way as it is done for the finite case, an orthogonal array with an infinite number of rows and columns, having a cyclic infinite automorphisms group.

Next we will show the first square matrix of order $20$ obtained from the $q_{i,j}$:

$$\left(
\begin{array}{cccccccccccccccccccc}
 0 & 0 & 0 & 0 & 0 & 0 & 0 & 0 & 0 & 0 & 0 & 0 & 0 & 0 & 0 & 0 & 0 & 0 & 0 & 0 \\
 0 & 1 & 2 & 3 & 4 & 5 & 6 & 7 & 8 & 9 & 10 & 11 & 12 & 13 & 14 & 15 & 16 & 17 & 18 & 19 \\
 0 & 2 & 1 & 5 & 7 & 3 & 10 & 4 & 13 & 15 & 6 & 18 & 20 & 8 & 23 & 9 & 26 & 28 & 11 & 31 \\
 0 & 3 & 5 & 1 & 9 & 2 & 7 & 11 & 4 & 8 & 16 & 6 & 23 & 20 & 28 & 24 & 10 & 34 & 30 & 12 \\
 0 & 4 & 7 & 9 & 1 & 6 & 2 & 5 & 3 & 11 & 14 & 23 & 19 & 26 & 8 & 29 & 24 & 10 & 27 & 38 \\
 0 & 5 & 3 & 2 & 6 & 1 & 11 & 16 & 18 & 21 & 26 & 9 & 4 & 32 & 34 & 42 & 40 & 7 & 15 & 8 \\
 0 & 6 & 10 & 7 & 2 & 11 & 1 & 9 & 20 & 3 & 23 & 8 & 13 & 4 & 29 & 36 & 38 & 26 & 45 & 44 \\
 0 & 7 & 4 & 11 & 5 & 16 & 9 & 1 & 15 & 2 & 29 & 28 & 32 & 25 & 3 & 13 & 45 & 42 & 6 & 24 \\
 0 & 8 & 13 & 4 & 3 & 18 & 20 & 15 & 1 & 7 & 33 & 2 & 9 & 40 & 16 & 34 & 5 & 11 & 44 & 49 \\
 0 & 9 & 15 & 8 & 11 & 21 & 3 & 2 & 7 & 1 & 42 & 26 & 5 & 35 & 33 & 4 & 54 & 50 & 22 & 6 \\
 0 & 10 & 6 & 16 & 14 & 26 & 23 & 29 & 33 & 42 & 1 & 3 & 36 & 2 & 11 & 8 & 12 & 5 & 4 & 35 \\
 0 & 11 & 18 & 6 & 23 & 9 & 8 & 28 & 2 & 26 & 3 & 1 & 39 & 22 & 25 & 14 & 51 & 40 & 55 & 15 \\
 0 & 12 & 20 & 23 & 19 & 4 & 13 & 32 & 9 & 5 & 36 & 39 & 1 & 3 & 54 & 60 & 28 & 22 & 2 & 58 \\
 0 & 13 & 8 & 20 & 26 & 32 & 4 & 25 & 40 & 35 & 2 & 22 & 3 & 1 & 57 & 45 & 6 & 41 & 17 & 66 \\
 0 & 14 & 23 & 28 & 8 & 34 & 29 & 3 & 16 & 33 & 11 & 25 & 54 & 57 & 1 & 6 & 13 & 36 & 64 & 2 \\
 0 & 15 & 9 & 24 & 29 & 42 & 36 & 13 & 34 & 4 & 8 & 14 & 60 & 45 & 6 & 1 & 49 & 2 & 57 & 80 \\
 0 & 16 & 26 & 10 & 24 & 40 & 38 & 45 & 5 & 54 & 12 & 51 & 28 & 6 & 13 & 49 & 1 & 74 & 69 & 9 \\
 0 & 17 & 28 & 34 & 10 & 7 & 26 & 42 & 11 & 50 & 5 & 40 & 22 & 41 & 36 & 2 & 74 & 6 & 68 & 18 \\
 0 & 18 & 11 & 30 & 27 & 15 & 45 & 6 & 44 & 22 & 4 & 55 & 2 & 17 & 64 & 57 & 69 & 68 & 1 & 95 \\
 0 & 19 & 31 & 12 & 38 & 8 & 44 & 24 & 49 & 6 & 35 & 15 & 58 & 66 & 2 & 80 & 9 & 18 & 95 & 1 \\
\end{array}
\right)$$

We report in Figure \ref{fig:row} the plots  of the first sequences in the infinite matrix, from the $3^{rd}$ row to the $7^{th}$ row and for $10^6$ columns of data obtained with Algorithm 1; the $1^{st}$ and $2^{nd}$  rows, which distribute along lines with slopes $0$ and $1$, respectively, are omitted because of their simplicity.

	\begin{figure}[H]
		\caption{First sequences in the infinite matrix, from the $3^{rd}$ row to the $7^{th}$ row (Algorithm 1)} \label{fig:row}
		\centering
		\begin{subfigure}{0.33\textwidth}
			\centering
			\caption{$3^{rd}$ row ($i=2$)}
				\includegraphics[width= \textwidth]{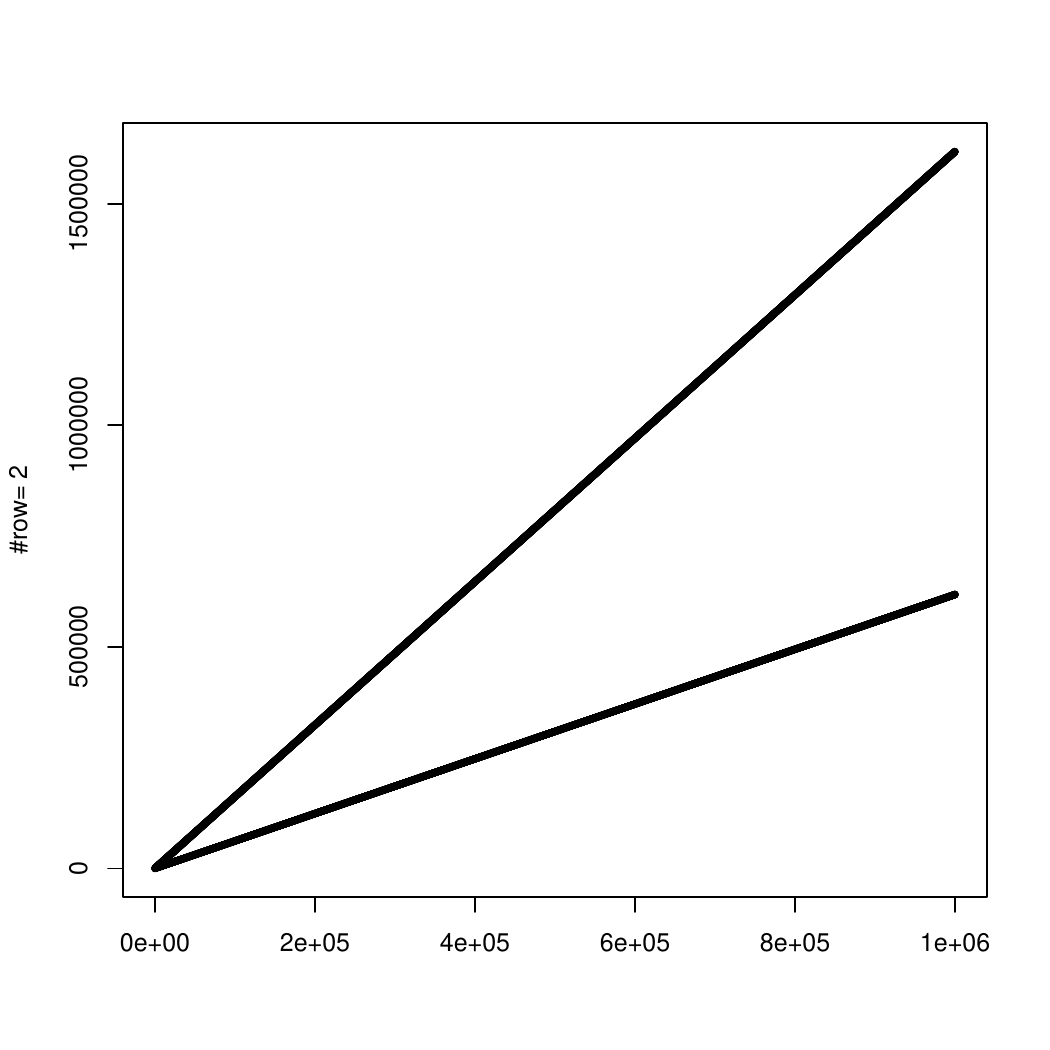}
		\end{subfigure}%
		\begin{subfigure}{0.33\textwidth}
			\centering
			\caption{$4^{th}$ row ($i=3$)}
			\includegraphics[width= \textwidth]{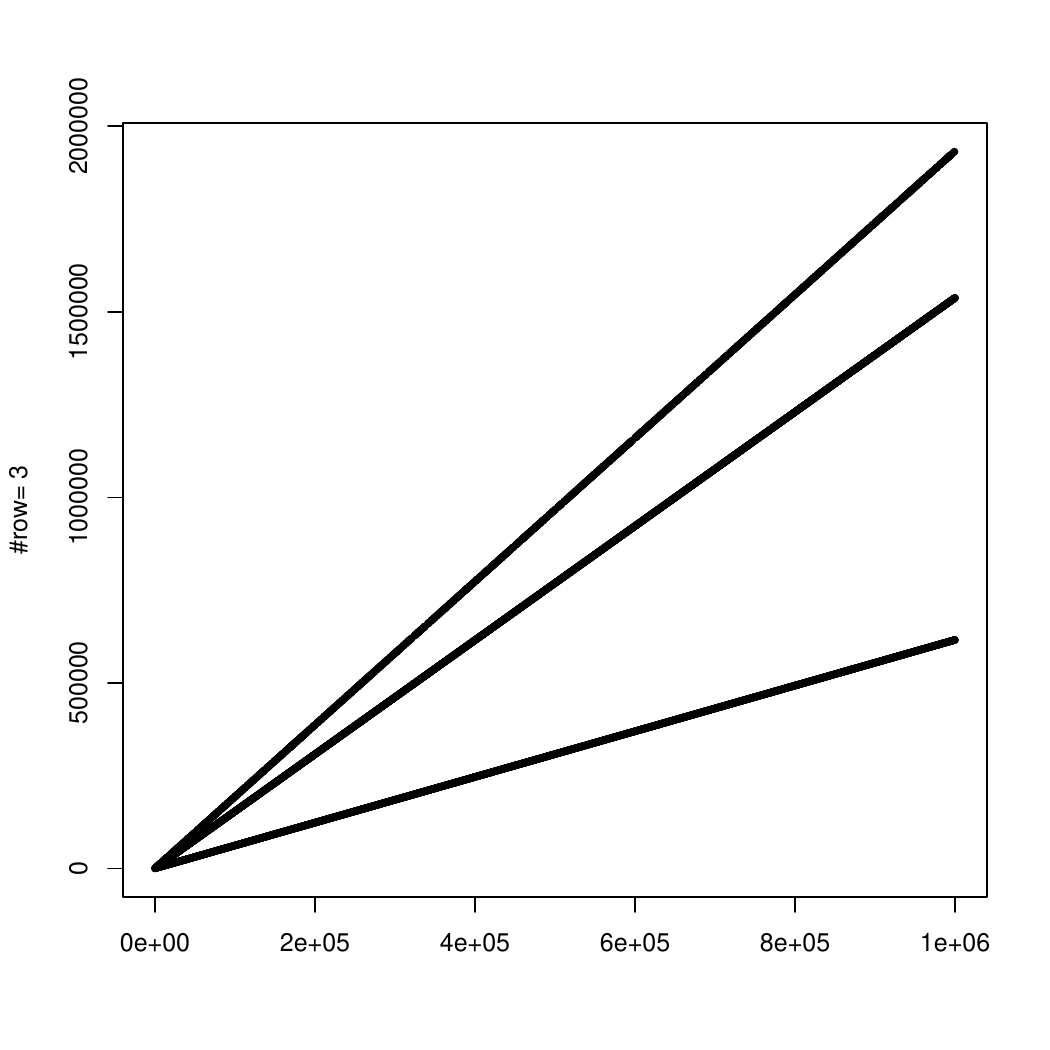}
		\end{subfigure}%
		\begin{subfigure}{0.33\textwidth}
			\centering
			\caption{$5^{th}$ row ($i=4$)}
			\includegraphics[width= \textwidth]{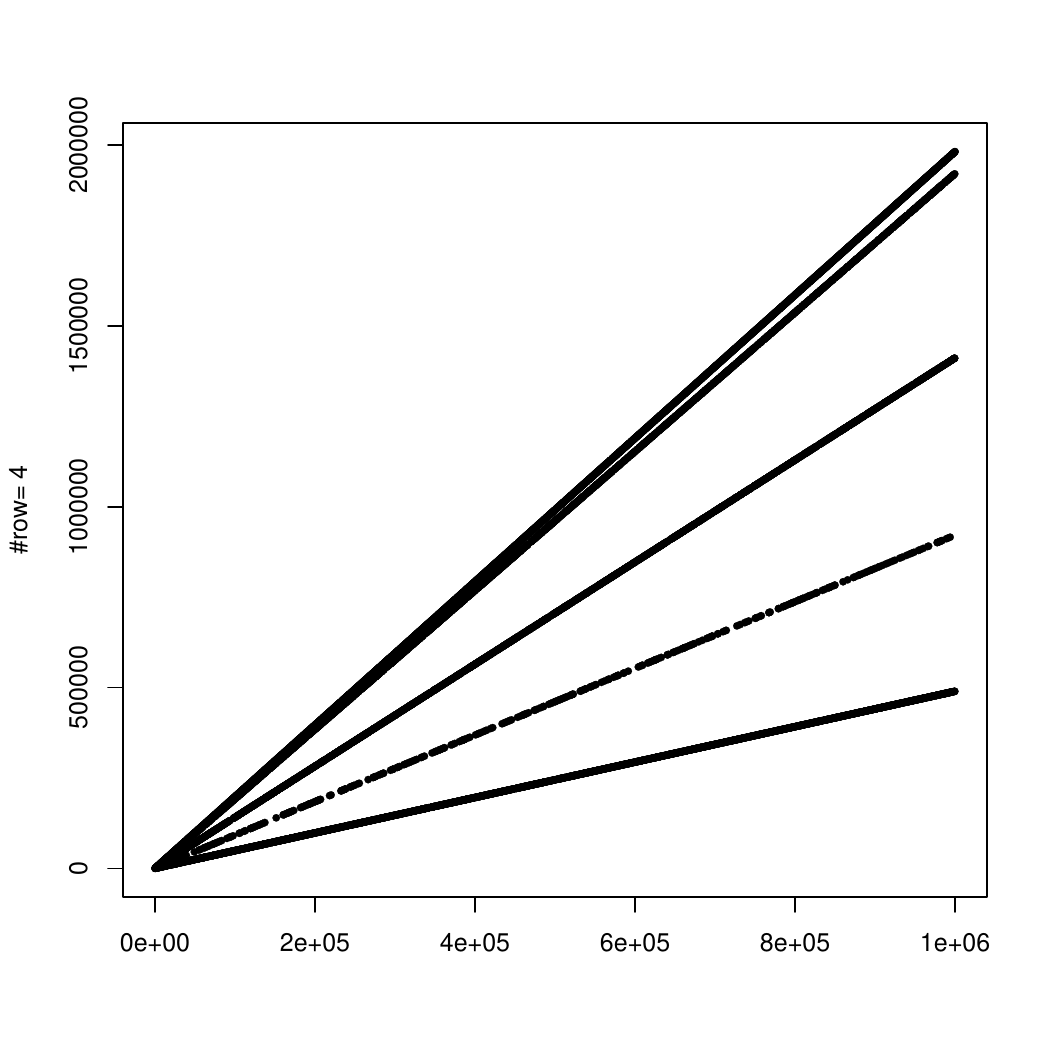}
		\end{subfigure}%

		\begin{subfigure}{0.33\textwidth}
			\centering
			\caption{$6^{th}$ row ($i=5$)}
				\includegraphics[width= \textwidth]{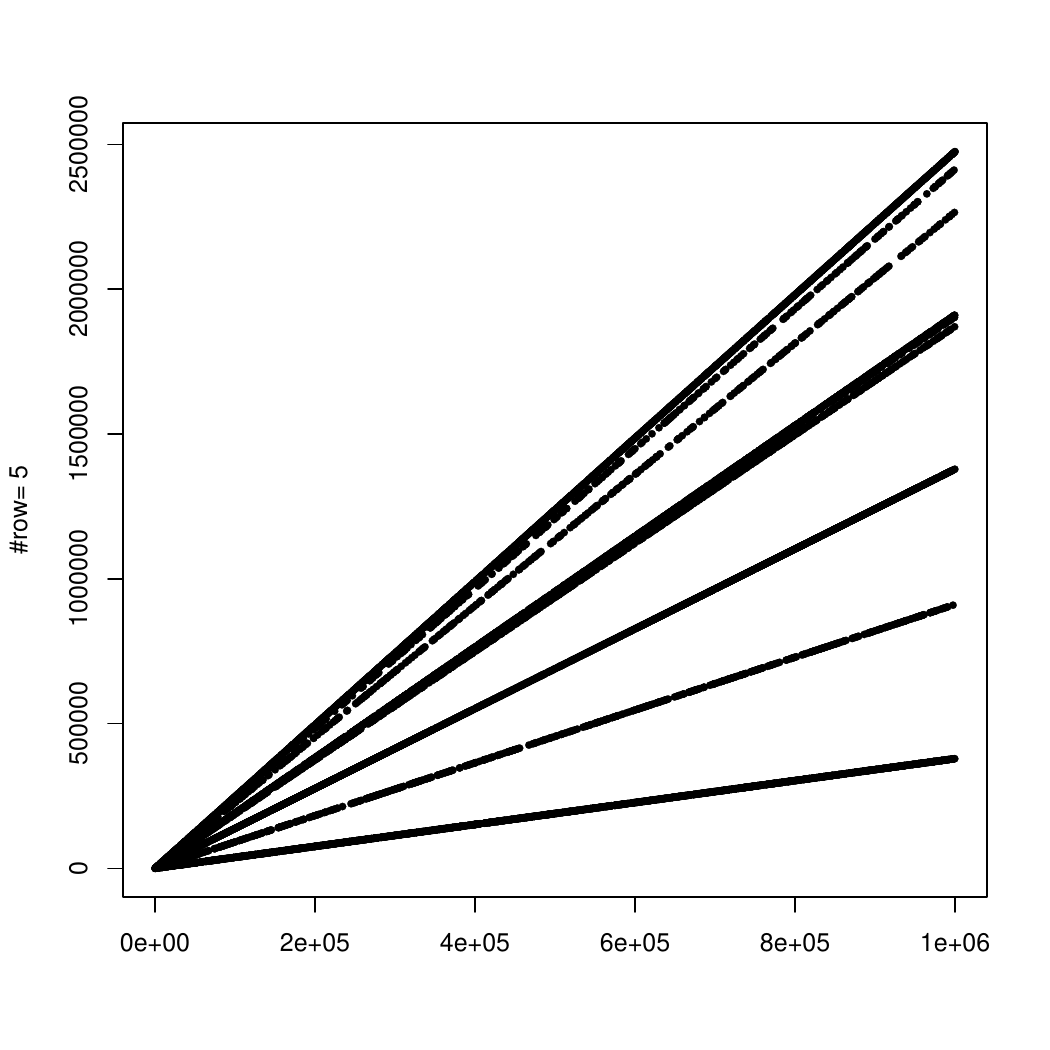}
		\end{subfigure}%
		\begin{subfigure}{0.33\textwidth}
			\centering
			\caption{$7^{th}$ row ($i=6$)}
	\includegraphics[width= \textwidth]{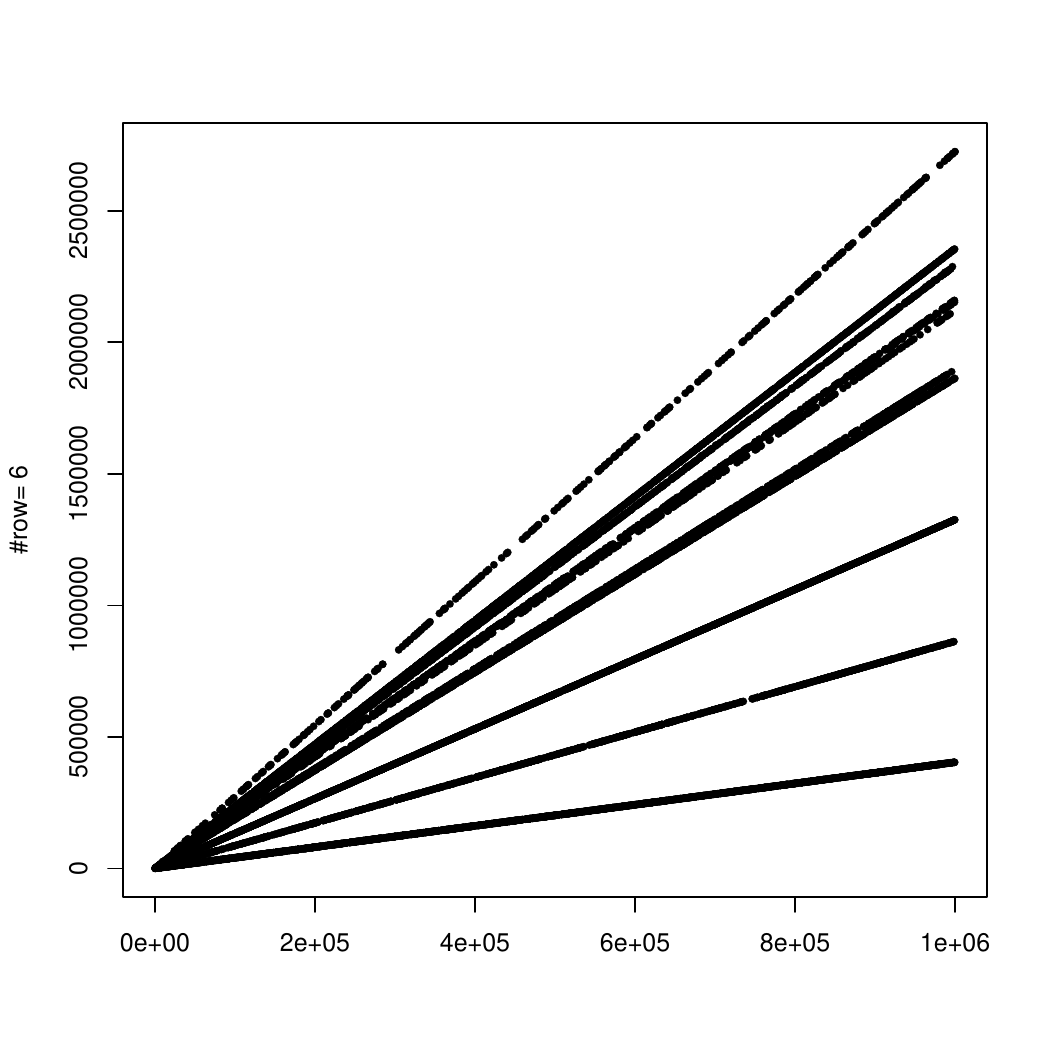}
	\end{subfigure}
		\begin{subfigure}{0.33\textwidth}
			\centering
			\caption*{} \end{subfigure}%
	\end{figure}

\section{Open questions and future research}\label{oq}

The main open question and object of future research is, of course, Conjecture \ref{conj}, for which there is strong numeric evidence in its favour, as we will show next. Of course, by the definition of Algorithm \ref{alg}, all the elements in each row are different non-negative numbers and also, when taking differences of elements in two different given rows from the second one onwards, all the differences give different integer numbers. The open question is that every non-negative integer number appears in each row and every integer number appears in each difference of rows. In the repository at \cite{ehubox} the reader can find the row data with one million points for the first seven rows obtained with Algorithm \ref{alg}. It can be found that every non-negative integer $\leq 378684$ appears in all rows and every integer with absolute value $\leq 78440$ appears as a difference in all difference of different rows (this, of course, does not means that only those numbers and differences appear, it is only intended to give an idea of consecutive values valid for the first six rows when sequences with one million values are considered). 

It would be good to obtain closed algebraic formulas for the rows from the fourth one onwards obtained with Algorithm \ref{alg}, in a similar way as we did for the third row. This seems to be a difficult and defiant task.

Looking at the figures at the end of Section \ref{ioas} it seems that the rows of the matrix $(q_{i,j})$ produced with Algorithm \ref{alg} are distributed asymptotically close to a set of lines. It would be interesting to study if this is really the case for all rows of the matrices.

Based on Figure \ref{fig:row} it might seem naked-eye that the number of lines follows the Fibonacci sequence. However, this seems not be the case, since the famous sequence  $1,1,2,3,5,8,13,\dots$ now turns out to be $1,1,2,3,5,9,12,\dots$. We have observed that after performing cluster analysis with the slopes in the last $10^4$ elements. The matrix data given by the slopes has been clustered by the hierarchical method, see  \cite{maimon}, where the euclidean distance and the complete-linkage clustering have been used for building the hierarchy of clusters, see \cite{everitt}.

	The complete-linkage clustering is a method of agglomerative hierarchical clustering. Initially each element is in a cluster of its own. Then, the two clusters separated by the shortest distance are combined. The clusters are sequentially combined into larger clusters until all elements end up being in the same cluster. The link between two clusters contains all element pairs, and the distance between clusters equals the distance between those two elements that are farthest away from each other. The shortest of these links that remains at any step causes the fusion of the two clusters whose elements are involved.
	Moreover, the Silhouette method has been used to validate the consistency within clusters of data, see \cite{rousseeuw},  the obtained curves are shown in Figure \ref{fig:sil} in Appendix A. We considered the $k$-means method for estimating the slopes for $k$ clusters, where $k$ correspond to the maximum in the Silhouette curve, see Table \ref{tb:alg1} in Appendix A. The $k$-means method aims to partition the values into $k$ groups such that the sum of squares from points to the assigned cluster centres is minimized. At the minimum, all cluster centres are at the mean of their Voronoi sets. The algorithm of Hartigan and Wong (1979) \cite{hartigan} is used by default. 

Besides knowing the number of lines for each sequence, it would be desirable to know also the slopes of the asymptotes. It is not clear even a so simple question as if they are algebraic numbers. Of course they are for the first three rows, but for the fourth row, estimations of the slopes of the three lines with $4$ precission digits are $0.6157$, $1.5376$, and $1.9331$ (for the estimation, the second half of the file \lq fourth\_row.txt\rq\ of the mentioned repository was analyzed, and the difference between the maximum and minimum was taken, being this difference less than $10^{-4}$), and they are not roots of polynomials of degree less than or equal to $3$ with integer coefficients with absolute value less than or equal to $5$, as one can see listing the roots of all those polynomials.

\section{Acknowledgments} We would like to thank Professor Jeffrey Shallit, from the University of Waterloo, for help with the use of the software Walnut \cite{walnut}.

 \appendix

\section{Cluster analysis} \label{sec:ap}

We report below the results of the cluster analysis explained in Section \ref{oq}, according to open questions and future research. The Silhouette curves of the hierarchical clustering  for $3^{rd}$ to $7^{th}$ rows are shown for Algorithm 1 in Figure \ref{fig:sil}, respectively. The results with the slopes estimations  are shown in Table \ref{tb:alg1} for Algorithm 1. Notice that the estimated values for the $3^{rd}$ row coincide, as expected, with  $\varphi^{-1}$ and $\varphi$ up to the sixth decimal digit for the proposed algorithm.

\begin{figure}[!ht]
	\caption{Silhouette curve of the hierarchical clustering for $3^{rd}$ to $7^{th}$ row (Algorithm 1)} \label{fig:sil}
	\centering
	\begin{subfigure}{0.33\textwidth}
		\centering
		\caption{$3^{rd}$ row ($i=2$)}
	\includegraphics[height=\ps \textheight]{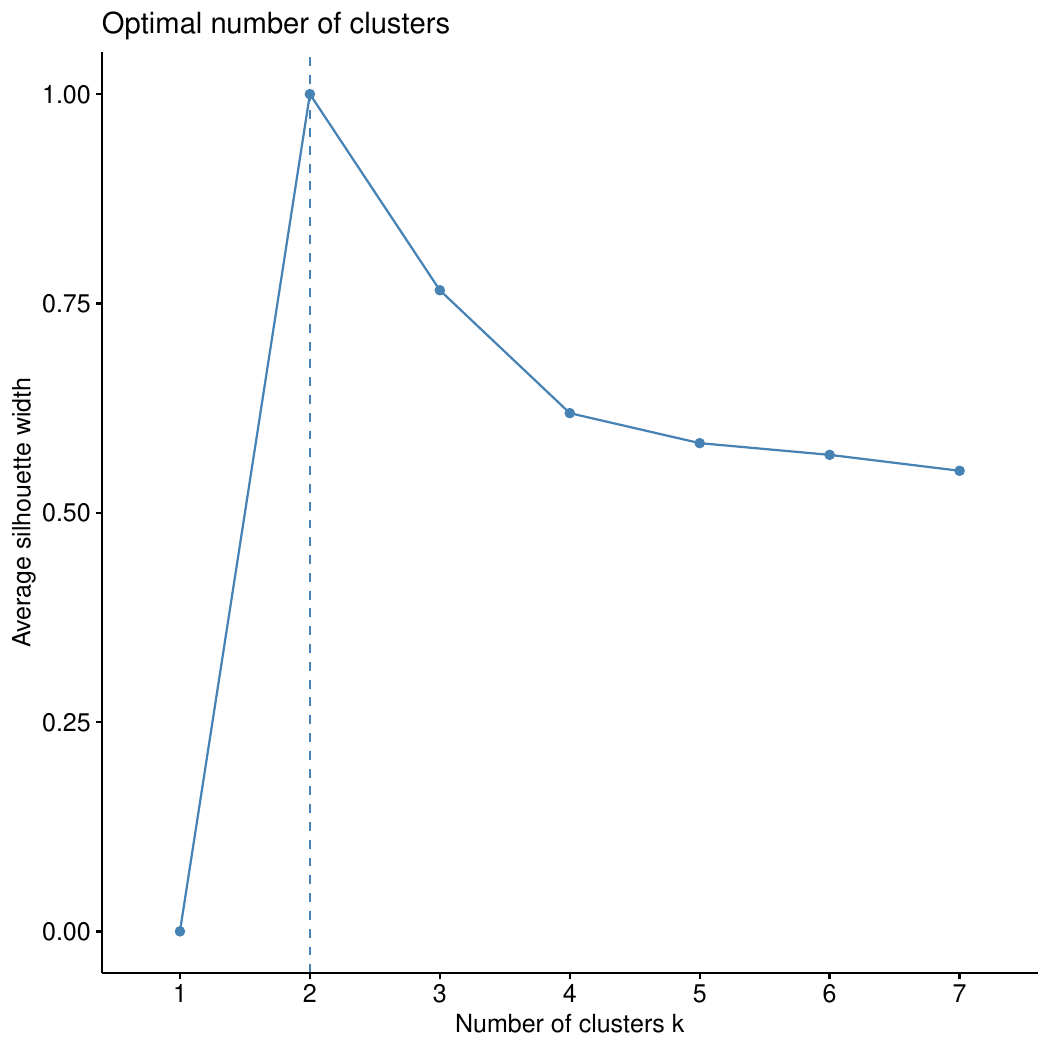}	
	\end{subfigure}%
	\begin{subfigure}{0.33\textwidth}
	\centering
	\caption{$4^{th}$ row ($i=3$)}
	\includegraphics[height=\ps \textheight]{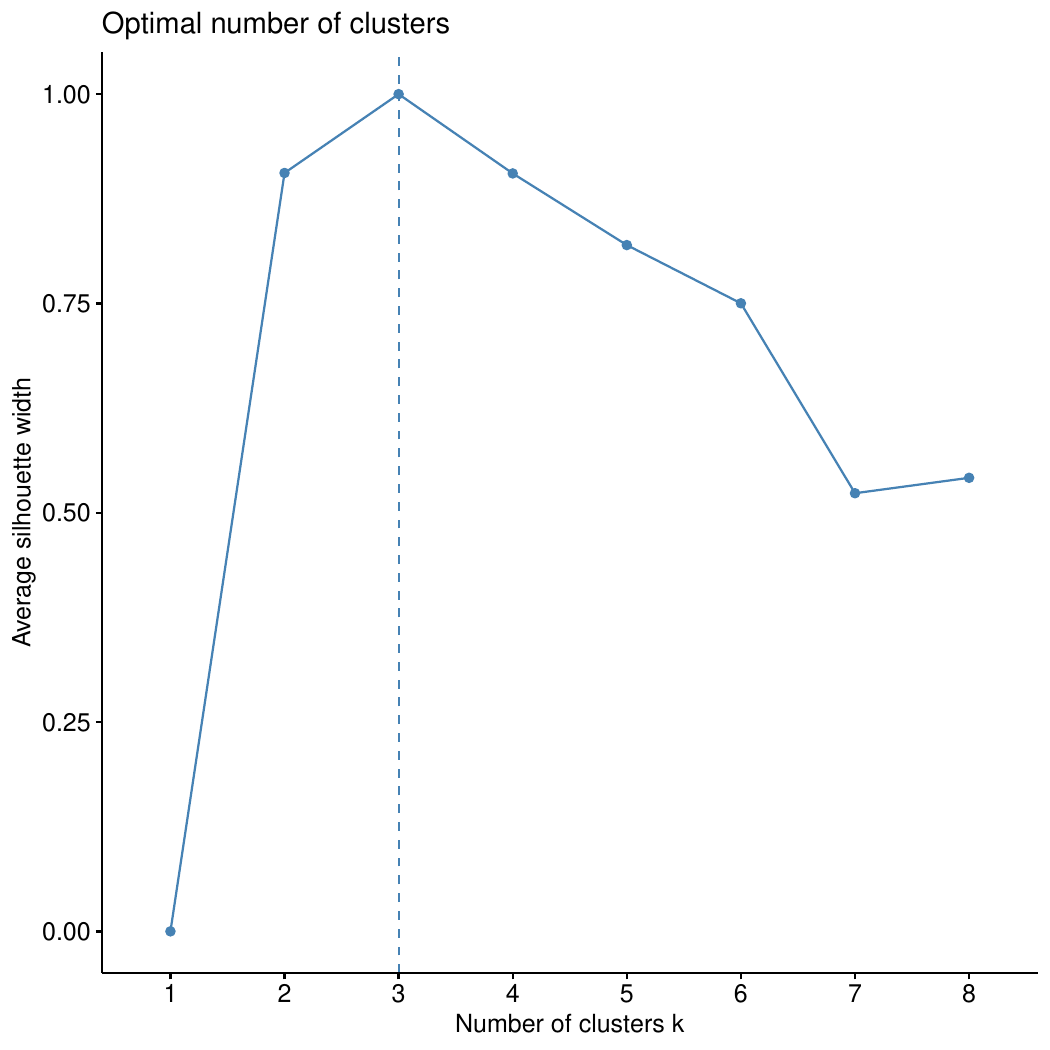}	
\end{subfigure}%
\begin{subfigure}{0.33\textwidth}
	\centering
	\caption{$5^{th}$ row ($i=4$)}
	\includegraphics[height=\ps \textheight]{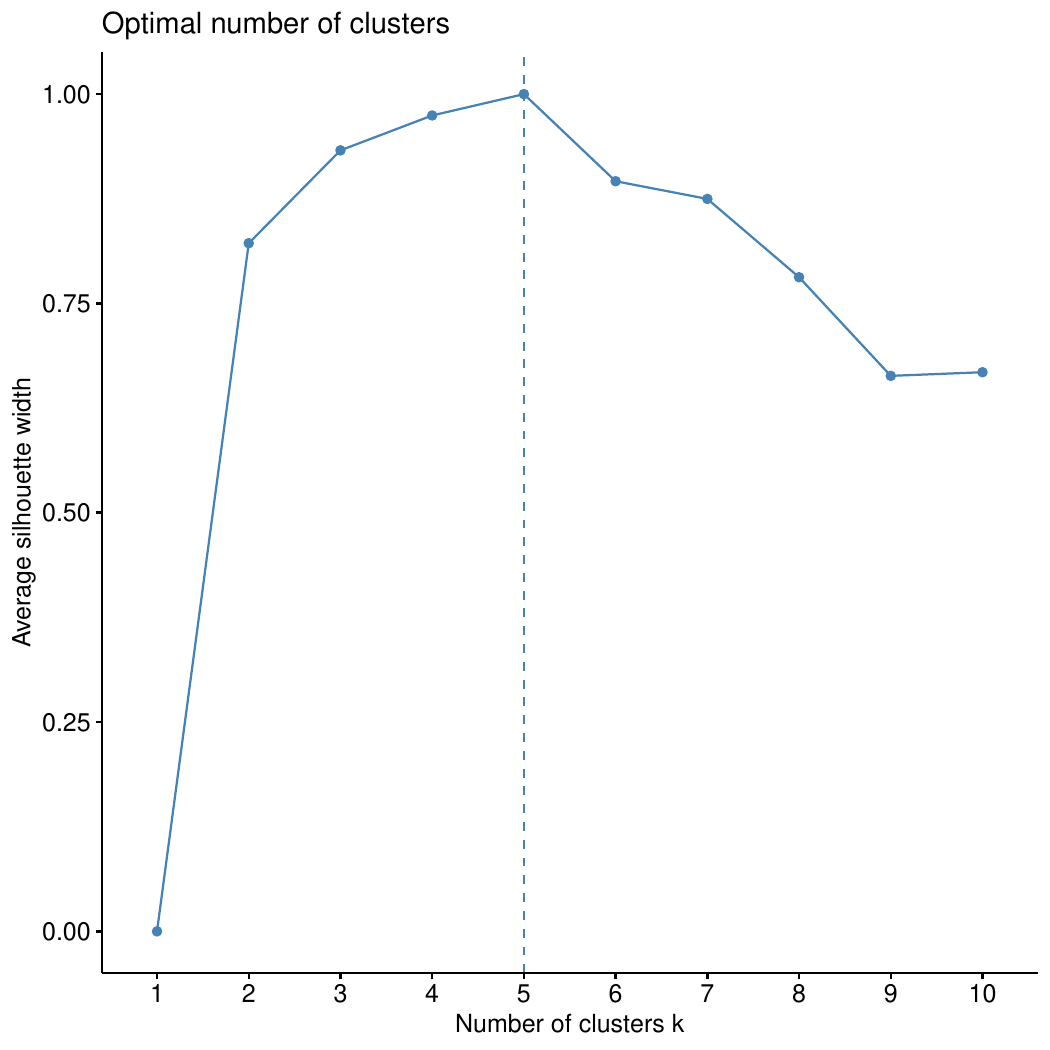}	
\end{subfigure}%

\begin{subfigure}{0.33\textwidth}
	\centering
	\caption{$6^{th}$ row ($i=5$)}
	\includegraphics[height=\ps \textheight]{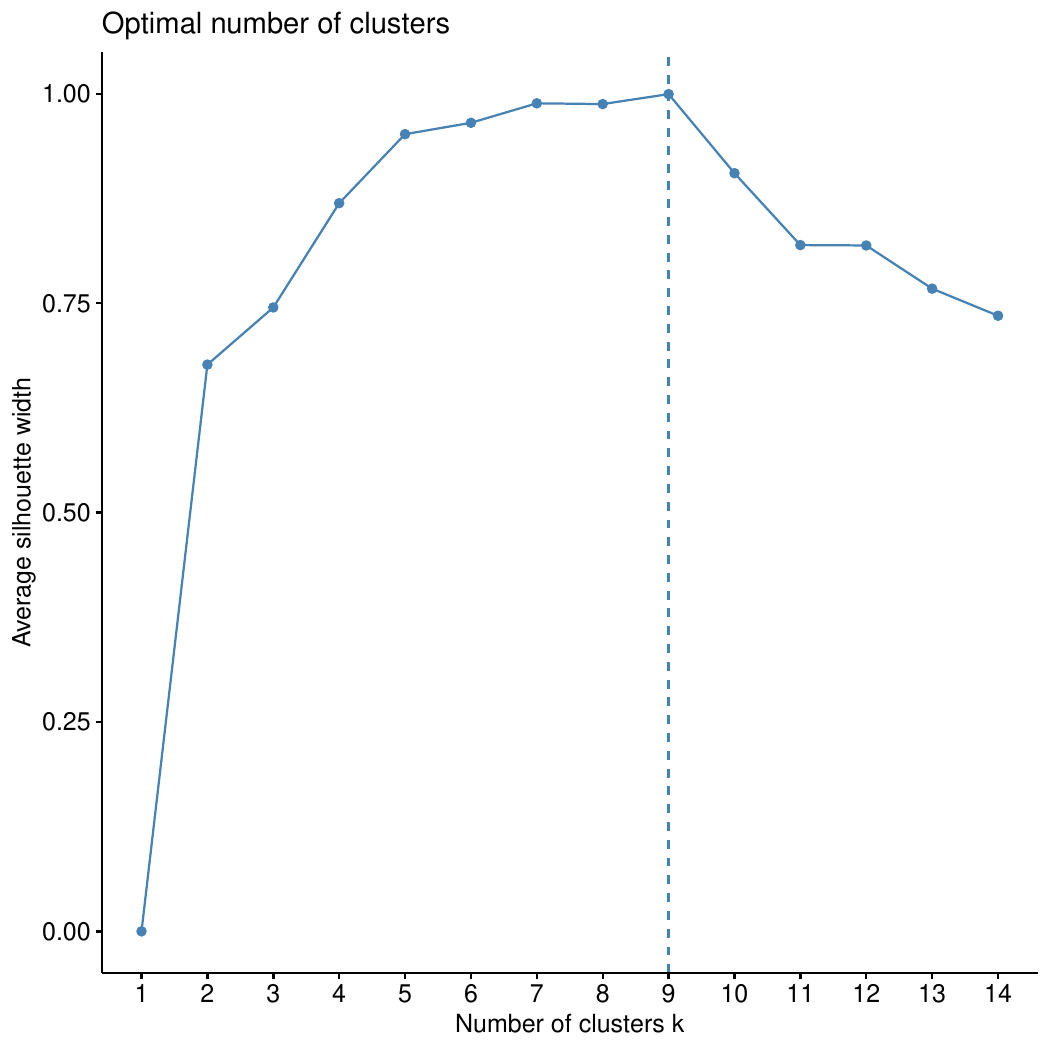}	
\end{subfigure}%
\begin{subfigure}{0.33\textwidth}
	\centering
	\caption{$7^{th}$ row ($i=6$)}
	\includegraphics[height=\ps \textheight]{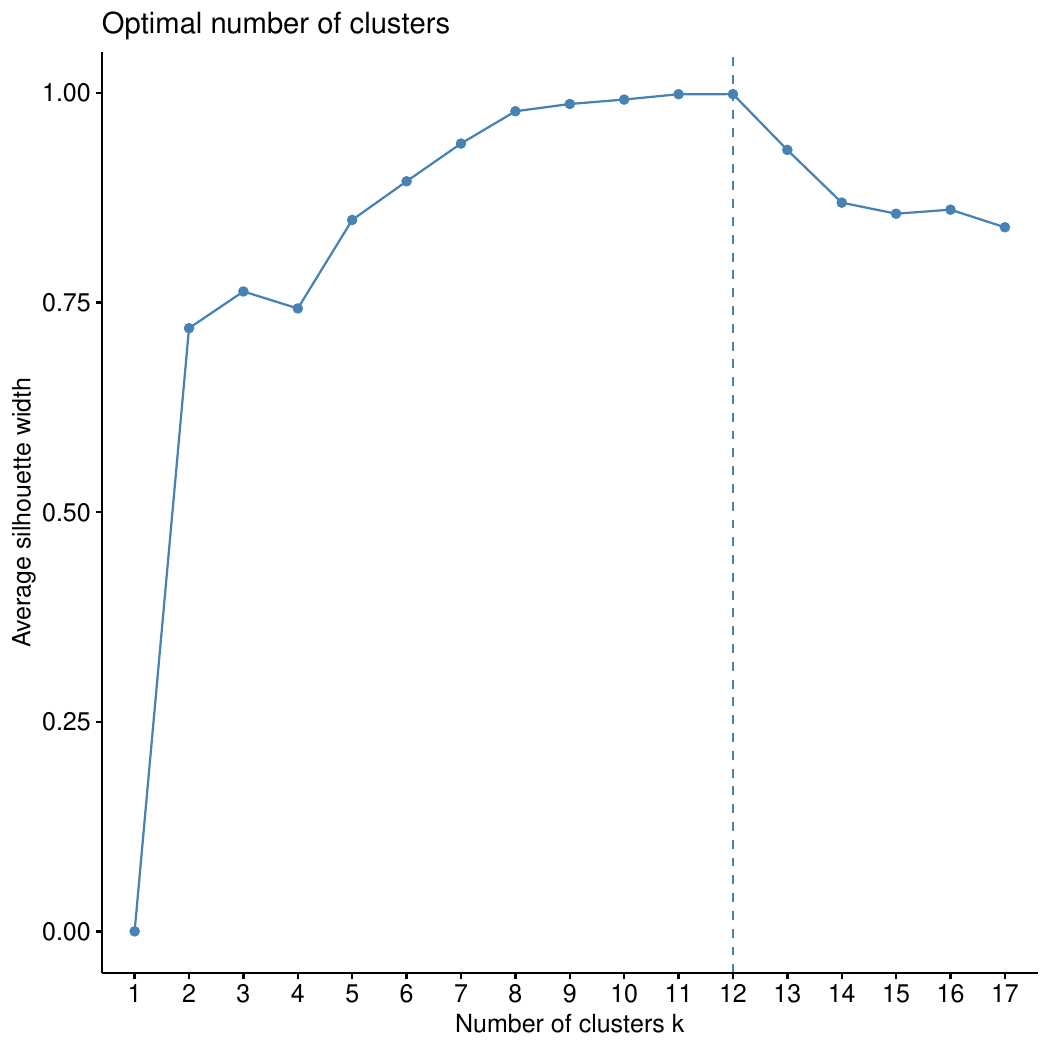}	
\end{subfigure}%
	\begin{subfigure}{0.33\textwidth}
	\centering
	\caption*{}	
	\end{subfigure}%
\end{figure}

\begin{table}[H]
	\centering
	\caption{Slopes for $2^{nd}$ to $7^{th}$ rows according to Algorithm 1 and $k$-means clustering, with data from the last $10^4$ elements and $k$ is the maximum in the Silhouette curve}  \label{tb:alg1}
	\begin{tabular}{lrrrrrr}
		\hline
	slopes	& $2^{nd}$ row & $3^{th}$ row & $4^{th}$ row & $5^{th}$ row & $6^{th}$ row & $7^{th}$ row \\ 
		\hline
		1 & 1 & 0.6180338 & 0.6157357 & 0.4897040 & 0.3786866 & 0.4033740 \\ 
		2 &  & 1.6180343 & 1.5376674 & 0.9215673 & 0.9111673 & 0.8630458 \\ 
		3 &  &  & 1.9331150 & 1.4116338 & 1.3786555 & 1.3255402 \\ 
		4 &  &  &  & 1.9215719 & 1.8712844 & 1.8630575 \\ 
		5 &  &  &  & 1.9823043 & 1.9096087 & 1.8959612 \\ 
		6 &  &  &  &  & 2.2667333 & 2.1209285 \\ 
		7 &  &  &  &  & 2.4143427 & 2.1577673 \\ 
		8 &  &  &  &  & 2.4750619 & 2.2913921 \\ 
		9 &  &  &  &  & 2.4750933 & 2.2914530 \\ 
		10 &  &  &  &  &  & 2.2949202 \\ 
		11 &  &  &  &  &  & 2.3556537 \\ 
		12 &  &  &  &  &  & 2.7247096 \\ 
		\hline
	\end{tabular}
\end{table}

\end{document}